\documentclass[11pt,reqno]{amsart}

\usepackage[margin=1in]{geometry}
\usepackage{amsfonts,amssymb}
\usepackage[initials]{amsrefs}
\usepackage{hyperref}
\usepackage{amsmath,verbatim,amsxtra}
\usepackage{amstext}
\usepackage{amsthm}
\usepackage{amscd}
\usepackage{mathrsfs}
\usepackage{mathtools}
\usepackage{bbm}
\usepackage{color}
\usepackage{aligned-overset}
\usepackage{enumerate}
\usepackage[normalem]{ulem}

\theoremstyle{plain}
\newtheorem{thm}{Theorem}
\newtheorem{prop}[thm]{Proposition}

\newtheorem{lem}[thm]{Lemma}
\newtheorem{clm}[thm]{Claim}

\theoremstyle{remark}
\newtheorem*{rmk}{Remark}
\newtheorem*{rmks}{Remarks}
\theoremstyle{definition}

\newenvironment{alist}
{

\begin{enumerate}}
{\end{enumerate}}

\newenvironment{rlist}
{

\begin{enumerate}}
{\end{enumerate}}

\newenvironment{Rlist}
{

\begin{enumerate}}
{\end{enumerate}}

\newenvironment{assume}
{

\begin{enumerate}}
{\end{enumerate}}

\renewcommand{\a}{{\alpha}}
\renewcommand{\b}{{\beta}}
\newcommand{\g}{{\gamma}}
\renewcommand{\d}{{\delta}}
\newcommand{\e}{{\varepsilon}}
\newcommand{\z}{{\zeta}}
\renewcommand{\t}{{\theta}}

\renewcommand{\l}{{\lambda}}

\newcommand{\G}{{\Gamma}}

\newcommand{\R}{{\mathbb R}}
\newcommand{\N}{{\mathbb N}}
\newcommand{\Z}{{\mathbb Z}}
\newcommand{\C}{{\mathbb C}}

\newcommand{\T}{{\mathbb T}}
\newcommand{\D}{{\mathbb D}}
\newcommand{\E}{\mathbb E}
\renewcommand{\P}{\mathbb P}

\newcommand{\cH}{{\mathcal H}}
\newcommand{\cK}{{\mathcal K}}

\newcommand{\cP}{{\mathcal P}}
\newcommand{\cN}{{\mathcal N}}

\newcommand{\cD}{\mathcal{D}}

\newcommand{\SU}{\mathrm{SU}}
\newcommand{\dd}{\mathrm{d}}

\DeclareMathOperator{\spn}{span}

\DeclareMathOperator{\sinc}{sinc}

\DeclareMathOperator{\Vr}{Var}

\DeclareMathOperator{\Exp}{Exp}
\DeclareMathOperator{\re}{Re}
\DeclareMathOperator{\GMC}{GMC}
\newcommand{\abs}[1]{\lvert#1\rvert}
\newcommand{\Abs}[1]{\left\lvert#1\right\rvert}
\newcommand{\norm}[1]{\lVert#1\rVert}

\newcommand{\wik}[1]{:\mathrel{#1}:}
\newcommand{\wikk}[2]{\ensuremath{:\mathrel{|#1|^{2#2}}:}}
\newcommand{\wikkk}[4]{\ensuremath{:\mathrel{{#1}^{#2}{\overline{#3}}^{#4}}:}}

\newcommand{\integllim}[3]{\int_{#1} \! #2 \, \mathrm{d}#3}
\newcommand{\integlim}[4]{\int_{#1}^{#2} \! #3 \, \mathrm{d}#4}
\newcommand{\ex}[1]{\E\left[#1\right]}
\newcommand{\var}[1]{\Vr\left[#1\right]}
\newcommand{\pr}[1]{\P\left[#1\right]}
\renewcommand{\Re}[1]{\re\left(#1\right)}

\title[Gaussian complex zeroes are not always normal]{Gaussian complex zeroes are not always normal:\protect\\ limit theorems on the disc}
\author{Jeremiah Buckley}
\author{Alon Nishry}

\address{J. Buckley, Department of Mathematics, King's College London, United Kingdom}
\email{jeremiah.buckley@kcl.ac.uk}
\address{A. Nishry, School of Mathematical Sciences, Tel Aviv University, Tel Aviv 69978, Israel}
\email{alonish@tauex.tau.ac.il.}

\thanks{The research of JB is supported in part by EPSRC New Investigator Award EP/V002449/1. The research of AN is supported in part by ISF Grant 1903/18.}


\begin{document}
\begin{abstract}
    We study the zeroes of a family of random holomorphic functions on the unit disc, distinguished by their invariance with respect to the hyperbolic geometry. Our main finding is a transition in the limiting behaviour of the number of zeroes in a large hyperbolic disc. We find a normal distribution if the covariance decays faster than a certain critical value. In contrast, in the regime of `long-range dependence' when the covariance decays slowly, the limiting distribution is skewed. For a closely related model we emphasise a link with Gaussian multiplicative chaos.
\end{abstract}
\maketitle
\section{Introduction}

\subsection{Statement of results}
We are interested in the zeroes of the random holomorphic functions
\begin{equation*}
    f_{0}(z)=\sum_{m=1}^{\infty}\frac{1} {\sqrt{m}}\z_{m}z^{m}\quad \text{ and }\quad f_{L}(z)=\sum_{m=0}^{\infty}\sqrt{\frac{\Gamma(L+m)} {\Gamma(L)m!}}\z_{m}z^{m}\quad\text{ for }L>0,
\end{equation*}
where $\{ \z_{m} \}$ is a sequence of iid $\cN_{\C}(0,1)$ standard complex Gaussians and $z$ belongs to the unit disc $\D$. The distribution of $f_L$ as a Gaussian analytic function (GAF) on $\D$ is determined by its covariance kernel
\begin{equation}\label{eq: cov}
    K_L(z,w)=\ex{f_L(z)\overline{f_L(w)}}=
    \begin{cases}
    (1-z\bar{w})^{-L}, &\text{if }L>0;\\
    \log\frac1{1-z\bar{w}}, &\text{if }L=0.
    \end{cases}
\end{equation}
A short computation of covariance kernels shows that if $\psi: \D \to \D$ is a disc automorphism then
\[
 f_0\circ\psi - (f_0\circ\psi)(0) \overset{d}{=} f_0\quad\text{ and}\quad f_L\circ\psi \cdot (\psi')^{L/2} \overset{d}{=} f_L \quad\text{ for }L>0,
\]
where $\overset{d}{=}$ denotes equality in distribution as Gaussian processes. Since $\psi'$ is a deterministic non-vanishing function, this means that for $L>0$ the zeroes of $f_L$ form a stationary point process in $\D$. Furthermore, if we fix the intensity of the zero process (equivalently $L$), that is, the mean number of zeroes per unit hyperbolic area, then $f_L$ is essentially the only GAF with this property. For further details see \cite{GAFbook}*{Chapter 2}. 

The functions $f_L$ have arisen in different contexts. Diaconis and Evans \cite{DE}*{Example 5.6} showed that the function $f_2$ (up to normalisation) arises as the limit of the logarithmic derivative of the characteristic polynomial of a random $n\times n$ unitary matrix, for large $n$. Peres and Vir\'{a}g \cite{PV} showed that the zeroes of $f_1$ form a determinantal process, and used this to describe statistical properties of the zero set. Chhaibi and Najnudel \cite{CN} recently showed a relation between the `boundary values' of the function $f_0$ and a certain limit of the circular $\b$ ensemble.

Let $n_{L}(r)$ be the number of zeroes of $f_{L}$ in the
disc $D(0,r)$ for $0<r<1$. In this article we will describe the fluctuations of $n_{L}(r)$ about its mean as $r\to1$. This mean can be computed via the Edelman-Kostlan formula, see \cite{GAFbook}*{Section~2.4}. For $L>0$, the asymptotic growth of the variance was studied in \cite{B} and one of the interesting features is a transition at the value $L=\frac12$. In this article we show that $n_{L}(r)$ satisfies a CLT for $L\ge\frac12$ while we find non-Gaussain behaviour for $L<\frac12$ (which we describe explicitly).

We note that (for the variance estimates, see \cite{B} for $L>0$ and Section~\ref{sec: var comput} for $L=0$)
\begin{equation}\label{eq: var est}
    \ex{n_L(r)}\overset{r\to1}{\sim}
    \begin{cases}
    \frac{1}{2(1-r)\log\frac1{1-r}}, &L=0;\\
    \frac{L}{2(1-r)}, &L>0,
    \end{cases}\:\text{and}\:
    \var{n_L(r)}\overset{r\to1}{\simeq}
    \begin{cases}
    \frac{1} {\left(1-r\right)^{2} \left(\log\frac{1}{1-r}\right)^{4}}, & L=0,\\
    \frac{1}{(1-r)^{2(1-L)}}, &L\in(0,\tfrac12);\\
    \frac{1}{1-r}\log\frac1{1-r}, &L=\frac12;\\
    \frac{1}{1-r}, &L>\frac12.
    \end{cases}
\end{equation}
In order to state our results we write
\begin{equation*}
    \hat{n}_{L}(r)=\frac{n_{L}(r)-\ex{n_{L}(r)}} {\sqrt{\var{n_{L}(r)}}}
\end{equation*}
for the normalised version of $n_{L}(r)$, put
\begin{equation}\label{eq: amL def}
    a_{m,0}=
    \begin{cases}0 & m=0,\\
    \frac{1}{m} & m\ge1,
    \end{cases}
\quad\text{ and }\quad a_{m,L}=\frac{\Gamma(L+m)} {\Gamma(L)m!} \quad\text{ for }L>0,
\end{equation}
and introduce the random variable $X_{L}=\sum_{m=0}^{\infty}a_{m,L}\left(\left|\z_{m}\right|^{2}-1\right)$ for $0\le L<\frac{1}{2}$. By Stirling's approximation we have $\frac{\Gamma(L+m)}{m!}\overset{m\to\infty}{\sim}m^{L-1}$
so that $\sum_{m=0}^{\infty}a_{m,L}^{2}<\infty$ for such $L$
and so the sum defining $X_{L}$ converges almost surely.
\begin{thm}\label{thm: main}\phantom{h}
\begin{rlist}
    \item \label{part: L ge 1/2} If $L\ge\frac{1}{2}$ is fixed, then we have $\hat{n}_{L}(r)\to\cN_{\R}\left(0,1\right)$ (the standard Gaussian) in law, as $r\to 1$.
    \item \label{part: L to half} If $L\to\frac{1}{2}$ and $r\to1$ simultaneously, then $\hat{n}_{L}(r)\to\cN_{\R}\left(0,1\right)$ in law.
    \item \label{part: 0<L<half} If $0\le L<\frac{1}{2}$ is fixed, then we have $\hat{n}_{L}(r)\to-c_{L}X_{L}$ in $L^{2}$ as $r\to1$ where
    \[
    c_{L}^2 = \left( \sum_{m=0}^\infty a_{m,L}^2 \right)^{-1} = \begin{cases} \frac{6}{\pi^2}, & L=0;\\
    \frac{\Gamma (1-L)^2}{\Gamma (1-2 L)}, &  0<L< \frac12.
    \end{cases}
    \]
\end{rlist}
\end{thm}
\begin{rmks}\phantom{h}
\begin{enumerate}
    \item The case $L=1$ of this theorem is \cite{PV}*{Corollary~3~(iii)}. It was proved using the determinantal structure, and so the methods do not apply to other values of $L$.
    
    \item For $L < \tfrac12$ the limit $X_L$ is determined by the `boundary values' of the process $f_L$, we shall elaborate on this remark in Section~\ref{sec: GMC}.
    \item Using Lyapunov's criterion, one can check that $\sum_{m=0}^{N}a_{m,L} \left(\left|\z_{m}\right|^{2}-1\right)$ obeys a CLT when $N\to\infty$, for $L\ge\frac{1}{2}$. This is essentially the reason for Gaussian behaviour when $L\to\tfrac12$.
\end{enumerate}
\end{rmks}

It is clear that $X_{L}$ is non-Gaussian, e.g., since $\ex{X_{L}^{3}} = 2\sum_{m=0}^{\infty} a_{m,L}^{3}\neq0$. In the case $L=0$ a direct computation using characteristic functions shows the limiting distribution is Gumbel, while for $0<L<\frac12$ we give sharp estimates on the decay of the tail probability. Since $\hat{n}_{L}(r)\to-c_{L}X_{L}$, this means that the `left' tail of $X_L$ corresponds to the `right' tail of $\hat{n}_L$ and vice-versa.
\begin{thm}\label{thm: tail}\phantom{a}
\begin{rlist}
    \item \label{thm: tail Gumbel part} $X_0$ is a Gumbel distributed random variable with mean $0$ and variance $\frac{\pi^2}{6}$.
    \item \label{thm: tail L>0 part} If $0<L<\frac12$ then $\P[X_L>x] = (\kappa_L +o(1)) e^{-x}$  and $\log\P[X_L<-x] = - (\l_L+o(1)) x^{1/L}$ as $x\to\infty$, where
    \begin{equation*}
        \kappa_L=\frac1e \prod_{m=1}^{\infty} \frac{e^{-a_{m,L}}}{1-a_{m,L}} \qquad\text{and}\qquad\l_L= L \, \Gamma(L)^{1/L} \left(-\sinc\left(\frac{\pi}{1-L}\right)\right)^{\frac{1}{L}-1}.
    \end{equation*}
\end{rlist}
\end{thm}
\begin{rmks}\phantom{w}
\begin{enumerate}
    \item We recall that the Gumbel CDF (with our normalisation) is $\exp(-e^{-x-\g_e})$ where $\g_e$ is Euler's constant. This means that $\P[X_0>x] \sim e^{-\g_e-x}$ as $x\to\infty$ and $\log\P[X_0<-x] = - e^{x-\g_e}$. The right tails of $X_L$ therefore have an exponential profile for all $L$, while the left tails are quite different. See Figure~\ref{fig: distributions} for an illustration of the PDF of $c_L X_L$ (that is, normalised to have mean $0$ and variance $1$).
    \item In Section~\ref{sec: GMC} we give a heuristic explanation for the appearance of the Gumbel distribution, using the theory of Gaussian multiplicative chaos.
\end{enumerate}

\end{rmks}

\begin{figure}
  \centering
  \includegraphics[width=0.8\columnwidth]{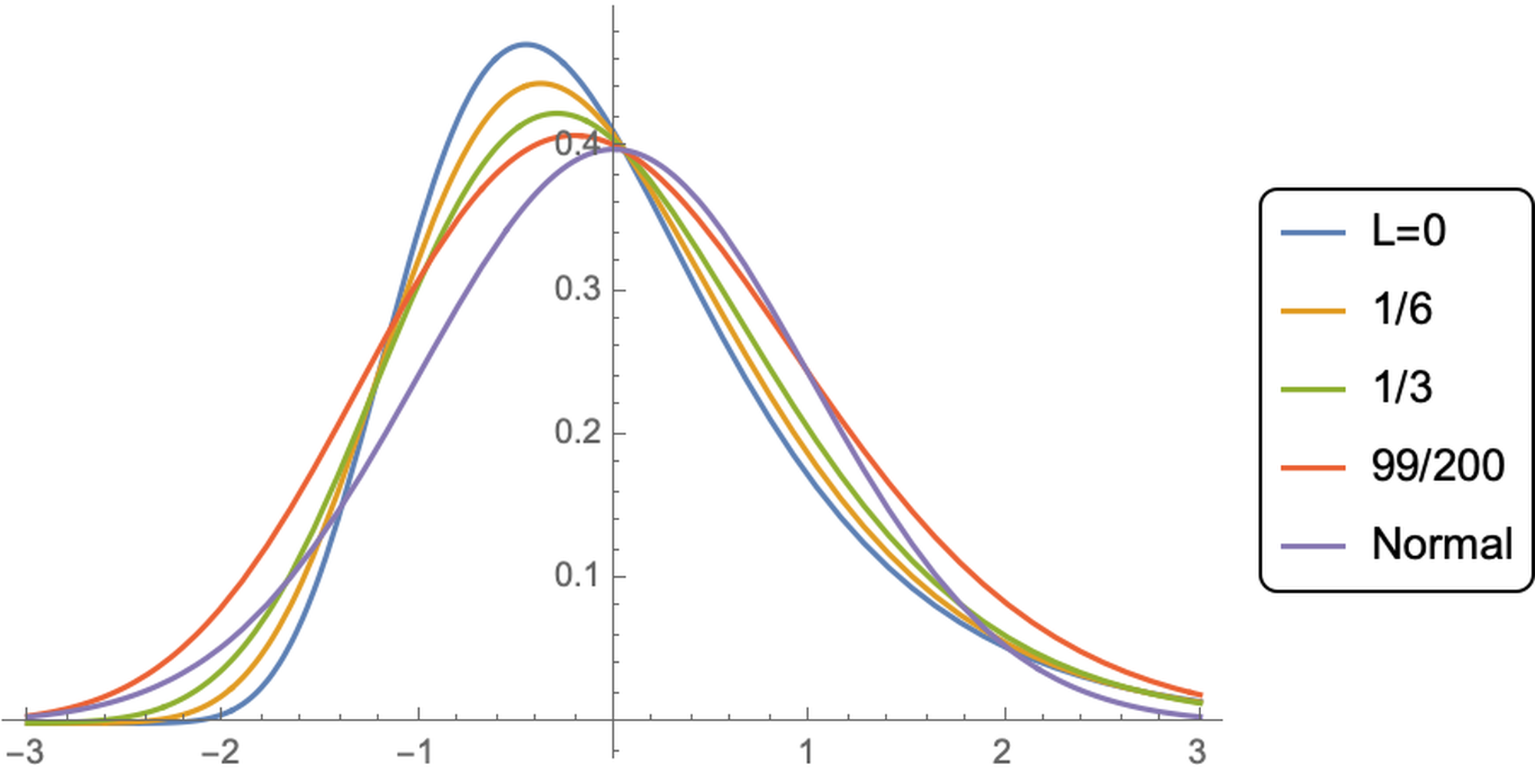}
  \caption{Distributions of the random variables $c_L X_L$ and the standard normal.}
  \label{fig: distributions}
\end{figure}

\subsection{Background and motivation}
The function $f_L$ was (to the best of our knowledge) first introduced by Leb\oe uf in \cite{Leb} where it is referred to as an ``analytic chaotic eigenstate''. It is viewed as a coherent state representation of a random quantum state. The unit disc is interpreted as the phase space of the corresponding quantum mechanical system, which is assumed to exhibit $\SU(1,1)$ symmetry. Furthermore, the fact that the coefficients $\z_m$ are complex valued reflects the absence of time-reversal symmetry. A different point of view, motivated by signal processing, is to consider $f_L$ as a Daubechies-Paul wavelet transform of white noise, see, e.g.,  \cite{BH}*{Theorem~2.3} and \cite{AHKR}*{Section~3}.

Leb\oe uf views the invariance of the zeroes of $f_L$ as a manifestation of ergodicity in phase space. The zero set is sometimes referred to as the `stellar representation' of the state (or `Majorana representation' see, e.g., \cite{Stellarbook}*{Chapter~7}) and is physically expected to determine the `Husimi function' $\frac{|f_L(z)|^2}{\ex{|f_L(z)|^2}}$, which gives the probability of finding a particle in a small neighbourhood of $z$. Interestingly \cite{PV}*{Theorem~6} gives an explicit formula to reconstruct $|f_L|$ from the random zeroes.

Another motivation for studying the zeroes of random holomorphic functions is to view the resulting point process as a system of interacting particles that exhibit local repulsion \cite{GAFbook}*{Chapter 1}. One fruitful approach is to compare and contrast the properties of different processes. It is particularly interesting to contrast $f_L$ with the `flat GAF' which is entire and invariant with respect to the Euclidean geometry \cite{GAFbook}*{Section~2.3}. Roughly speaking, the zeroes of $f_L$ for large $L$ behave like the flat zeroes, in contrast for small $L$ one expects to see `genuine hyperbolic phenomena'. Asymptotic normality of the zeroes of the flat GAF was described by Sodin-Tsirleson and Nazarov-Sodin \cites{NS,NS2,ST1}.

In the Euclidean setting there are many similarities between the zeroes of the `flat GAF' and the infinite Ginibre ensemble, see for example \cite{GN}. It is therefore also natural to compare the behaviour of the zeroes of $f_L$ with the determinantal process $\mathcal{X}_L$ with kernel
\begin{equation*}
    \cK_L(z,w)=\frac L\pi \frac1{(1-z\overline{w})^{L+1}}
\end{equation*}
and reference measure $\mathrm{d}\mu_L(z)=(1-|z|^2)^{\frac12(L-1)}\mathrm{d}m(z)$; here $m$ is the Lebesgue measure. Krishnapur \cite{KrishPhd}*{Theorem~3.0.5} showed that $\mathcal{X}_L$ are the only determinantal processes on the disc with analytic kernel that are invariant with respect to the automorphisms, and the intensity of the point process $\mathcal{X}_L$ and the zeroes of $f_L$ is the same. For $L=1$ the \emph{processes} are the same \cite{PV}. Moreover, it is also shown there that if $L\neq1$ then the zeroes of $f_L$ do not have a determinantal structure. Kartick Adhikari (private communication) has shown that there is no transition in the behaviour of the variance for the determinantal models and by \cite{GAFbook}*{Theorem~4.6.1} a CLT holds for all $L$. See also the recent work of Fenzl-Lambert \cite{FL}*{Section~2.3}.

\subsection{Related work}
Consider a real-valued stationary Gaussian sequence $(X_n)_{n\in\Z}$ with covariance kernel $r(n)=\ex{X_nX_0}$ which decays like $n^{-\a}$ for large $n$. Let $H$ be a function of Hermite rank $k$, that is, we can expand $H=\sum_{n\ge k}c_n H_n$ in terms of the Hermite polynomials in an appropriate sense. Consider the random variable
\begin{equation*}
    Y_N=\sum_{n=1}^N H(X_n).
\end{equation*}
By results of Breuer-Major and Dobrushin-Major \citelist{\cite{BM}*{Theorem~1}\cite{BM}*{Theorem~1'}\cite{DM}*{Theorem~1}}:
\begin{itemize}
    \item If $\a>\frac1k$ then the variance of $Y_N$ grows linearly with $N$ and a CLT holds. 
    \item If $\a=\frac1k$ then the variance of $Y_N$ grows at the rate $N\log N$, but a CLT still holds. 
    \item If $\a<\frac1k$ then the variance of $Y_N$ grows at the rate $N^{2-\a k}$, and a non-CLT holds. 
\end{itemize}
One may also consider the number of zeroes in the interval $[0,T]$ of a real-valued stationary Gaussian process $f:\R\to\R$, as $T\to\infty$. For sufficiently fast decay of the covariance, combining results of Cuzick and Slud, one gets a CLT for the zeroes \citelist{\cite{Cuz}*{Theorem~1}\cite{Slud91}*{Theorem~3}}; Slud found a non-Gaussian limit for a family with long range dependence \cite{Slud94}*{Theorem~3.2}. In our setting the random variable $n_L(r)$ has Hermite rank $2$, see Proposition~\ref{prop: Wiener expan}, and large values of $L$ correspond to fast decay of the covariance, see \eqref{eq: cov}.

A related problem in higher dimensions is the study of the nodal (i.e., zero) sets of random Laplace eigenfunctions; we refer the interested reader to the survey \cite{R} and the references therein. For example, Marinucci, Rossi and Wigman found that a CLT holds on the sphere  \cite{MRW}*{Corollary~1.3} while, in contrast, the same authors with Peccati showed that the fluctuations on the torus are non-Gaussian \cite{MPRW}*{Theorem~1.1}.

Curiously, when studying the pair correlations in the circular $\b$ ensemble, Aguirre, Soshnikov and Sumpter \cite{ASS}*{Theorem~2.1} discovered a non-Gaussian limit that is similar in form to the $X_L$ appearing in our Theorem~\ref{thm: main}.

\subsection{Links with Gaussian multiplicative chaos (GMC) when \texorpdfstring{$L=0$}{L=0}}\label{sec: GMC}

The random Fourier series
\begin{equation*}
    \sum_{m=0}^{\infty}\sqrt{a_{m,L}}\z_{m}e^{im\t}
\end{equation*}
does not converge to a function, but can be understood mathematically as a random distribution (i.e., generalised function). Such an object is sometimes referred to as a (complex) $1/f^{\a}$ noise on the unit circle; here $\a=1-L$ since $a_{m,L}\overset{m\to\infty}{\simeq}m^{L-1}$ \cite{Mand}. Heuristically we can think of this noise as representing the `boundary values' of $f_L$ on the unit circle, or conversely we can regard $f_L$ as the Poisson extension of the noise on the unit circle to the interior of the disc (see \cite{Johnson}*{Theorem 4} for a deterministic statement). The $\a=1$ (aliter $L=0$, so-called pink noise) case is particularly interesting and there is an extensive literature (in both mathematics and physics) on log-correlated processes. This theory, moreover, has links to random matrix theory and conjecturally with number theory \cite{P}*{Section~4}. We shall only touch on a small part of the theory here.

For $0\le L<\frac12$, a careful examination of the proof of Theorem~\ref{thm: main} shows that $\hat{n}_L(r)$ can be approximated in $L^2$ by the random variable
\[
-c_L r^{2} \sum_{m=0}^{\infty} a_{m,L} \left(\left|\zeta_{m}\right|^{2}-1\right) r^{2m}
\]
which converges to $-c_L X_L$ in $L^2$, as $r\to1$. On the other hand it is easy to compute
\begin{equation*}
    \int_{-\pi}^{\pi}\left(\big|f_{L}(re^{i\theta})\big|^{2}-\ex{\big|f_{L}(re^{i\t})\big|^{2}}\right) \frac{\dd\theta}{2\pi}=\sum_{m=0}^{\infty} a_{m,L} \left(\left|\zeta_{m}\right|^{2}-1\right) r^{2m}
\end{equation*}
and so we may think of $X_L$ as representing the integral\footnote{Fyodorov and Keating studied the integral defining $X_0$ in a different context \cite{FyK}*{Section~3~(d)}.}
\begin{equation*}
    \int_{-\pi}^{\pi}\left(\big|f_{L}(e^{i\theta})\big|^{2}-\ex{\big|f_{L}(e^{i\t})\big|^{2}}\right) \frac{\dd\theta}{2\pi},
\end{equation*}
bearing in mind that $f_L$ is properly a generalised function on the unit circle.

We now restrict to the case $L=0$. Write $u_0=\Re{f_0}$ and notice that
\[
 \frac{1}{2}\ex{u_{0}^{2}(z)}=\ex{\abs{f_{0}(z)}^{2}}.
\]
Consider, for $0<r<1$, the measures on the unit circle defined by 
\[
\dd\GMC_{r}^{\gamma}(\t)= \exp\left(\g  u_{0} (re^{i\t}) -\frac{\g^{2}}{2} \ex{u_{0}^{2}(re^{i\theta})} \right) \frac{\dd \t}{2\pi}.
\]
The weak limit (as $r\to1$) of this sequence of measures, denoted by $\GMC^{\gamma}$, is the Gaussian multiplicative chaos with coupling coefficient $0<\gamma<1$; it is a singular continuous random measure. For a comprehensive introduction to the theory we refer the reader to the survey \cite{RV}.

Curiously it turns out that one can derive Theorem~\ref{thm: tail}~\eqref{thm: tail Gumbel part} via the theory of GMC. We will only give a heuristic explanation. We are not aware of any way to extend this to $L>0$. 
\begin{prop}[The Fyodorov-Bouchaud formula (Remy, Chhaibi-Najnudel) \citelist{\cite{FyB} \cite{Remy}*{Theorem~1.1}\cite{CN}*{Corollary~2.5}}]
For $\gamma\in\left(0,1\right)$ the law of the total mass of the GMC is given
by
\[
\GMC^{\gamma}\left(\T\right)\overset{d}{=}K_{\gamma}\mathbf{e}^{-\gamma^{2}},
\]
where $K_{\gamma}=\Gamma\left(1-\gamma^{2}\right)^{-1}$ and $\mathbf{e}$
is a standard exponential random variable.
\end{prop}

We may think of $\GMC^{\gamma}$ as a sort of generating function
for $X_0$. Expanding $\GMC^{\g}$ in powers of $\g$ we find
\begin{align*}
\GMC^{\gamma}(\T) & =1+\gamma\int_{-\pi}^{\pi} u_{0}(e^{i\t}) \frac{\dd\theta}{2\pi} + \frac{\gamma^{2}}{2} \int_{-\pi}^{\pi} \left(u_{0}^2(e^{i\t}) - \ex{u_{0}^{2}(e^{i\t})}\right) \frac{\dd\theta}{2\pi} + \dots\\
& =1 + \g\cdot 0 +\gamma^{2}X_{0}+\dots
\end{align*}
and thus
\[
X_{0}\overset{d}{=}\lim_{\g\to0} \frac{K_{\g}\mathbf{e}^{-\g^{2}}-1} {\g^{2}}.
\]
To see how this leads to the Gumbel distribution we compute
\begin{equation*}
    \pr{\frac{ K_{\gamma} \mathbf{e}^{-\g^2} -1} {\g^2}\le t}=
    \pr{\mathbf{e}\ge\left( \frac{K_\g} {1+\g^2 t} \right)^{1/\g^2} } =\exp\left( -\left( \frac{K_{\gamma}} {1+\g^2 t} \right)^{1/\g^2} \right).
\end{equation*}
Taking the limit $\gamma\to0$ we find (here $\gamma_{e}$ is Euler's constant)
\[
\pr{X_{0} \le t}=\exp\left(-e^{-t-\gamma_{e}}\right)
\]
which is the Gumbel CDF.

\bigskip
The paper is organised as follows. In Section~\ref{sec: overview} we give an outline of the method. In Section~\ref{sec: L>1/2} we prove Theorem~\ref{thm: main} for $L>\frac12$. In Section~\ref{sec: var comput} we compute the asymptotic growth of the variance of $n_0(r)$. In Section~\ref{sec: L<1/2}  we complete the proof of Theorem~\ref{thm: main}. In Section~\ref{sec: tail} we prove Theorem~\ref{thm: tail}.

We conclude the introduction with a word on notation. We write $A\lesssim B$ if there exists a constant $C$, independent of the relevant variables, such that $A\le CB$. We write $A\simeq B$ if $A\lesssim B$ and $B\lesssim A$. We write $A=O(B)$ if $|A|\lesssim B$. We write $A\sim B$ if $A/B\to1$ when we take an appropriate limit.

\subsection*{Acknowledgements}
Yan Fyodorov explained the links between log-correlated processes and the Gumbel distribution to us. We had a number of useful discussions with Avner Kiro about the asymptotics of the variance appearing in Section~\ref{sec: var comput}, and with Igor Wigman about the subtleties of the fourth moment method. We thank the anonymous referees for their careful reading of the manuscript and numerous suggestions which substantially improved the readability.

\section{Outline of the method}\label{sec: overview}
Our investigations centre on the Wiener chaos expansion (sometimes called the Hermite-It\={o}
expansion) of the random variable $n_L(r)$. This expansion is well-known to experts, and appears implicitly in the papers \cites{B,NS,ST1}. In order to state it we first introduce some notation. 

Let $\mathrm{d}\mu(\z)=\tfrac{1}{\pi}e^{-|\z|^{2}}\mathrm{d}m(\z)$ denote the Gaussian measure on the plane (here $m$ is the planar Lebesgue measure) and write $\cP_{q}$ for the polynomials (in the variables $\z$ and $\bar{\z}$) of degree at most $q$ considered as subspace of $L^{2}(\mu)$. Denote by $\cH^{\wik{0}}=\cP_{0}$ and $\cH^{\wik{q}}=\cP_{q}\ominus\cP_{q-1}$ for $q\geq1$ (here $\ominus$ denotes orthogonal complement). Given a monomial $\z^{\a}\bar{\z}^{\b}$ with $\a+\b=q$ we write $\wikkk{\z}{\a}{\z}{\b}$ to denote its projection to $\cH^{\wik{q}}$, which is usually called a Wick product (a complex Hermite polynomial of degree $\a+\b$).

We now state the expansion, for completeness we include more details and a proof in Appendix~\ref{app: chaos}.
\begin{prop}\label{prop: Wiener expan}
Write $\widehat{f}_L\left(z\right)=\frac{f_L\left(z\right)}{K_L\left(z,\bar{z}\right)^{1/2}}$ and define
\begin{equation*}
    n_L(r;\a)=\frac{\left(-1\right)^{\a+1}}{\a\left(\a!\right)}\frac{1}{2\pi i}\integlim{\partial D\left(0,r\right)}{}{\frac{\partial}{\partial z}\wikk{\widehat{f}_L\left(z\right)}{\a}}{z}.
\end{equation*}
Then $n_L(r;\a)$ belongs to the $2\a$-th component of the Wiener chaos corresponding to $f_L$ and 
\begin{equation*}
    n_L(r)-\ex{n_L(r)}=\sum_{\a=1}^\infty n_L(r;\a)
\end{equation*}
where the sum converges in $L^2$.
\end{prop}
Let us indicate a heuristic explanation of the expansion. A computation (see \cite{Jan}*{Example 3.32}) shows that the set of all Wick products $\wikkk{\z}{\a}{\z}{\b}$ with $\a+\b=q$ is an orthogonal basis for $\cH^{\wik{q}}$, and moreover $\norm{\wikkk{\z}{\a}{\z}{\b}}^{2}=\a!\b!$
(the norm here is the norm inherited from $L^{2}(\mu)$). Furthermore
\cite{Jan}*{Theorem 2.6}
\[
L^{2}(\mu)=\bigoplus_{q=0}^{\infty}\cH^{\wik{q}}.
\]
We expand the logarithm with respect to this orthonormal basis and a calculation \cite{NS}*{Lemma 2.1} yields
\begin{equation}\label{eq: log expan}
    \log\left|\z\right|^{2}=-\g_e+\sum_{\a=1}^{\infty}\frac{(-1)^{\a+1}}{\a(\a!)}\wikk{\z}{\a}
\end{equation}
where the equality holds in $L^{2}\left(\mu\right)$.

From the argument principle and direct computation we have
\[
n_L\left(r\right)=\frac{1}{2\pi i}\integlim{\partial D\left(0,r\right)}{}{\frac{f'_L\left(z\right)}{f_L\left(z\right)}}{z}=\frac{1}{2\pi i}\integlim{\partial D\left(0,r\right)}{}{\frac{\partial}{\partial z}\log\left|f_L(z)\right|^{2}}{z}
\]
and the Edelman-Kostlan formula \cite{GAFbook}*{Section~2.4} gives
\begin{equation*}
n_L\left(r\right)-\E\left[n_L\left(r\right)\right]=\frac{1}{2\pi i}\integlim{\partial D\left(0,r\right)}{}{\frac{\partial}{\partial z}\log \Abs{\widehat{f}_L(z)}^{2}}{z}.
\end{equation*}
Inserting \eqref{eq: log expan} into this expression and exchanging
the sum with the derivative and the integral formally yields the expansion given in the proposition. Furthermore, the orthogonality of the Wick products yields the orthogonality of $n_L\left(r;\a\right)$ for different values of $\a$.

Let us now outline how we use the expansion to prove the main theorem. The orthogonality of the expansion allows us to compute
\begin{equation*}
    \var{n_L(r)}=\sum_{\a=1}^\infty\ex{n_L(r;\a)^2}
\end{equation*}
and we will show that if $0\le L\le\frac12$ then
\begin{equation*}
    \var{n_L(r)}\overset{r\to1}{\sim}\ex{n_L(r;1)^2},
\end{equation*}
that is, only the first (non-trivial) component of the chaos contributes. In contrast, if $L>\frac12$ then all of the terms of the sum are of comparable size. In this latter case we show that each of the terms $n_L(r;\a)$ is asymptotically normal, through the method of moments. This idea goes back to \cite{ST1}, although the scheme developed there and modified in \cite{BS} only works for $L>1$, essentially due to the slower decay of the covariance kernel in the hyperbolic setting. Instead we use the Fourth Moment Theorem \cite{fourth mom}, a powerful method for proving a CLT for random variables that belong to a fixed component of the Wiener chaos.

If $0\le L \le \frac12$ then, by the variance estimates just mentioned, we have $n_L(r)=n_L(r;1)+o_\P(1)$. We analyse the expression for $n_L(r;1)$ in detail: for $L<\frac12$ we show that it is asymptotic to $X_L$ (when normalised properly), in contrast for $L=\frac12$ we show that it is asymptotically normal (the transition is essentially down to the  summability of $a_{m,L}^2$).    

\section{Proof of the CLT for \texorpdfstring{$L>\frac{1}{2}$}{L>1/2}}\label{sec: L>1/2}

In this section we will prove Theorem~\ref{thm: main}~\eqref{part: L ge 1/2} in the case that $L>\frac12$. The method is fairly standard, relying on the fourth moment theorem, and we accordingly do not give all of the details. The idea is to replace $n_{L}\left(r\right)$ by the random variable
\[
n_{L}^{\left(M\right)}(r)=\sum_{\a=1}^{M} n_{L}(r;\a)
\]
and to prove that:
\begin{Rlist}
    \item \label{CLT: L2 bound} There exists $r_{0}<1$ such that
    \begin{equation*}
        \ex{\left(n_{L}(r)-\ex{n_{L}(r)}-n_{L}^{\left(M\right)}(r)\right)^{2}}\le \frac{C_{L}}{\sqrt{M}}\var{n_{L}\left(r\right)}
    \end{equation*}
    for all $r_{0}\le r<1$ and $M\ge1$.
    \item \label{CLT: AN for finite m} For each fixed $M$
    \begin{equation*}
        \frac{n_{L}^{\left(M\right)}(r)} {\sqrt{\var{n_{L}^{\left(M\right)}(r)}}} \to \cN_{\R}\left(0,1\right)
    \end{equation*}
    in distribution, as $r\to1$.
\end{Rlist}
The result then follows; for completeness we prove this in Appendix~\ref{app: CLT}.
\subsection{Some preliminary calculations}
In order to implement the strategy outlined above, we will need the following lemma, which uses the notation $\widehat{K}_{L}(z,w)=\frac{K_{L}(z,w)}{\sqrt{K_{L}(z,z) K_{L}(w,w)}}$.
\begin{lem}\label{lem: 2nd/4th moment chaos}
  If $\a\ge1$ then
  \begin{align*}
      \ex{\big(n_{L}(r;\a)\big)^{2}} & = \left(\frac{1}{2\pi i}\right)^{2} \frac{1}{\a^{2}} \iint\limits_{\partial D\left(0,r\right)^{2}} \frac{\partial^{2}} {\partial z\partial w} \left|\widehat{K}_{L}\left(z,w\right)\right|^{2\alpha} \mathrm{d}z \mathrm{d}w\\
      &= \frac{L^{2}r^{4}} {2\pi\left(1-r^{2}\right)^{2}} \integlim{-\pi}{\pi} {\Abs{ \frac{1-r^{2}} {1-r^{2}e^{i\theta}} }^{2\a L} \Abs{\frac{1-e^{i\theta}} {1-r^{2}e^{i\theta}}}^{2} } {\theta}
  \end{align*}
  and
  \[
  \ex{\big(n_{L}(r;\alpha)\big)^{4}} = \left(\frac{1}{2\pi i}\right)^{4} \frac{1}{\alpha^{4}\left(\alpha!\right)^{4}} \int_{\partial D\left(0,r\right)^{4}} \frac{\partial^{4}} {\partial z_{1}\dots\partial z_{4}} \ex{\prod_{j=1}^{4}\wikk{\widehat{f}\left(z_{j}\right)} {\alpha} }\prod_{j=1}^{4}\mathrm{d}z_{j}.
  \]
\end{lem}
We postpone the proof of the lemma to Appendix~\ref{app: chaos}, since it simply involves exchanging expectation with integrals and derivatives. We will also need the following estimate.
\begin{lem}\label{lem: bound integral 2nd moment}
  If $L>\frac12$ is fixed then
  \begin{equation*}
      c_{L,\a} (1-r)\le \integlim{-\pi}{\pi} {\Abs{ \frac{1-r^{2}} {1-r^{2}e^{i\theta}} }^{2\a L} \Abs{\frac{1-e^{i\theta}} {1-r^{2}e^{i\theta}}}^{2} } {\theta} \le C_L\alpha^{-3/2}(1-r) + \big(C (1-r)\big)^{2\a L}
  \end{equation*}
  for $r\ge r_0$ and $\a\ge1$.
\end{lem}
\begin{proof}
An easy computation yields $\left|1-r^{2}e^{i\t}\right|^{2}=\left(1-r^{2}\right)^{2}+2r^{2}\left(1-\cos\theta\right)$
and so we get
\begin{equation*}
\integlim{-\pi}{\pi} {\left| \frac{1-r^{2}} {1-r^{2}e^{i\theta}} \right|^{2\alpha L} \left|\frac{1-e^{i\theta}} {1-r^{2}e^{i\theta}} \right|^{2}} {\theta}
=4\integlim{0}{\pi} {\left(1+ \frac{2r^{2}} {\left(1-r^{2}\right)^{2}} \left(1-\cos\theta\right)\right)^{-(\alpha L+1)} \frac{1-\cos\theta} {(1-r^{2})^{2}}}{\theta}.
\end{equation*}

We separate the `small' and `big' values of $\theta$. The small values contribute (below $B$ denotes the beta function and $\e>0$ is small but fixed)
\begin{align}\label{eq: small est alph var}
    \integlim {0}{\varepsilon} {\Big(1+ \frac{2r^{2}} {\left(1-r^{2}\right)^{2}} &\left(1-\cos\theta\right)\Big)^{-(\alpha L+1)} \frac{1-\cos\theta} {(1-r^{2})^{2}}} {\theta} \lesssim \integlim{0}{\varepsilon} {\left(1+\frac{c\t^{2}}{(1-r^{2})^{2}}\right)^{-(\alpha L+1)} \frac{\theta^{2}}{\left(1-r^{2}\right)^{2}}}{\theta}\notag\\
    \overset{x=\frac{\sqrt{c}\theta}{1-r^{2}}} &{\lesssim} (1-r^{2}) \integlim{0}{\infty} {\left(1+x^{2}\right)^{-(\alpha L+1)} x^{2}}{x}
    \overset{y=\frac{1}{1+x^{2}}}{\simeq}(1-r) \integlim{0}{1} {y^{\alpha L-\frac{3}{2}} \sqrt{1-y}} {y}\notag\\
    &= (1-r) B\Big(\alpha L-\frac{1}{2},\frac{3}{2}\Big)
    \simeq (1-r) \frac{\G(\a L- \frac{1}{2} )} {\G (\alpha L+1)}    \le C_L\alpha^{-3/2}(1-r).
\end{align}
The remaining contribution is
\begin{equation*}\label{eq: large est alpha var}
\integlim {\varepsilon} {\pi} {\left(1+ \frac{2r^{2}} {(1-r^{2})^2} (1-\cos\t)\right)^{-\alpha L} \frac{2\left(1-\cos\theta\right)}{\left(1-r^{2}\right)^{2}+2r^{2}\left(1-\cos\theta\right)}}{\theta}
\le \big(C (1-r)\big)^{2\a L} \integlim{\varepsilon}{\pi}{}{\theta}.
\end{equation*}

In the other direction, arguing similarly to \eqref{eq: small est alph var}, we get
\begin{equation*}
    \integlim {0}{\varepsilon} {\Big(1+ \frac{2r^{2}} {\left(1-r^{2}\right)^{2}} \left(1-\cos\theta\right)\Big)^{-(\alpha L+1)} \frac{1-\cos\theta} {(1-r^{2})^{2}}} {\theta} \gtrsim (1-r) \integlim{\frac12}{1} {y^{\alpha L-\frac{3}{2}} \sqrt{1-y}} {y}
\end{equation*}
where the lower limit in the last integral arises from choosing $r_0$ appropriately.
\end{proof}
\subsection{Proof of \texorpdfstring{\eqref{CLT: L2 bound}}{(I)}}
Using Proposition~\ref{prop: Wiener expan} and Lemma~\ref{lem: 2nd/4th moment chaos} we get
\begin{align*}
    \ex{\left(n_{L}\left(r\right)-\ex{n_{L}\left(r\right)} -n_{L}^{\left(M\right)}(r)\right)^{2}} & = \sum_{\a>M} \ex{\left(n_{L}\left(r;\a\right)\right)^2}\\
    &=\frac{L^{2}r^{4}}{2\pi\left(1-r^{2}\right)^{2}} \sum_{\a>M} \integlim{-\pi}{\pi} {\Abs{ \frac{1-r^{2}} {1-r^{2}e^{i\theta}} }^{2\a L} \Abs{\frac{1-e^{i\theta}} {1-r^{2}e^{i\theta}}}^{2} } {\theta}.
\end{align*}
Using the estimates in Lemma~\ref{lem: bound integral 2nd moment} we conclude that
\begin{align*}
\ex{\left(n_{L}(r) - \ex{n_{L}(r)}-n_{L}^{(M)}(r)\right)^{2}} &\le \frac{C_L}{(1-r)^{2}} \sum_{\alpha>M} \left[\alpha^{-3/2}(1-r) + \big(C (1-r)\big)^{2\a L}\right]\\
&\le C_L\left[\frac{1}{\sqrt{M}(1-r)} + \big(C (1-r)\big)^{2(M-1)L} \right]\\
&\le \frac{C_L}{\sqrt{M} (1-r)}\le \frac{C_L}{\sqrt{M}} \var{n_L(r)}
\end{align*}
uniformly in $M$, where the last bound follows from \eqref{eq: var est}. This completes
the proof of \eqref{CLT: L2 bound}.

\subsection{Proof of \texorpdfstring{\eqref{CLT: AN for finite m}}{(II)}}

We wish to show that the sum
\[
n_{L}^{\left(M\right)}(r)=\sum_{\alpha=1}^{M}n_{L}\left(r;\alpha\right)
\]
is asymptotically normal, and so it suffices to see that the vector
\[
\begin{pmatrix}n_{L}\left(r;1\right)\\
\vdots\\
n_{L}\left(r;M\right)
\end{pmatrix}
\]
satisfies a multi-variate CLT. By the multi-dimensional fourth moment
theorem \cite{fourth mom}*{Theorem 1} it is enough to check
that
\[
\frac{\ex{\big(n_{L}(r;\alpha)\big)^{4}}} {\ex{\big(n_{L}(r;\alpha)\big)^{2}}^{2}}\to3
\]
as $r\to1$, for each fixed $\a$. We recall that, by Lemma~\ref{lem: 2nd/4th moment chaos}, we have
\[
\ex{\big(n_{L}(r;\alpha)\big)^{4}} = \left(\frac{1}{2\pi i}\right)^{4} \frac{1}{\alpha^{4}\left(\alpha!\right)^{4}} \int_{\partial D\left(0,r\right)^{4}} \frac{\partial^{4}} {\partial z_{1}\dots\partial z_{4}} \ex{\prod_{j=1}^{4}\wikk{\widehat{f}\left(z_{j}\right)} {\alpha} }\prod_{j=1}^{4}\mathrm{d}z_{j}.
\]

Let $\mathcal{D}=\mathcal{D}(\a)$ denote the set of (bipartite) graphs with $8\alpha$ vertices such that:
\begin{itemize}
\item For each $1\le j\le4$ there are $\alpha$ vertices labelled $j$
and $\alpha$ vertices labelled $\overline{j}.$
\item Each vertex has degree exactly $1$, i.e., every vertex is paired with exactly one other vertex.
\item Each edge joins a vertex labelled $j$ to a vertex labelled $\overline{k}$
for $j\neq k$.
\end{itemize}
Now if the edge $e$ joins a vertex labelled $j$ to a vertex labelled
$\overline{k}$ then we write $\widehat{K}_L(e) = \widehat{K}_L(z_{j,}z_{k})$
and we define the value of a graph $\g\in\cD$ to be
\[
v\left(\gamma\right)=\prod_{e}\widehat{K}_L(e).
\]
By \cite{Jan}*{Theorem 3.12}
\[
\ex{\prod_{j=1}^{4}\wikk{\widehat{f}\left(z_{j}\right)}{\alpha} } =\sum_{\gamma\in\mathcal{D}}v\left(\gamma\right)
\]
and so
\[
\ex{\big(n_{L}(r;2\alpha)\big)^{4}} = \left(\frac{1}{2\pi i}\right)^{4} \frac{1}{\a^{4}(\a!)^{4}} \sum_{\gamma\in\mathcal{D}} \int_{\partial D\left(0,r\right)^{4}} \frac{\partial^{4}}{\partial z_{1}\dots\partial z_{4}} v(\g) \prod_{j=1}^{4}\mathrm{d}z_{j}.
\]
\begin{figure}
  \centering
  \includegraphics[trim={0 18.3cm 1.5cm 1.6cm},clip,width=0.6\columnwidth]
  {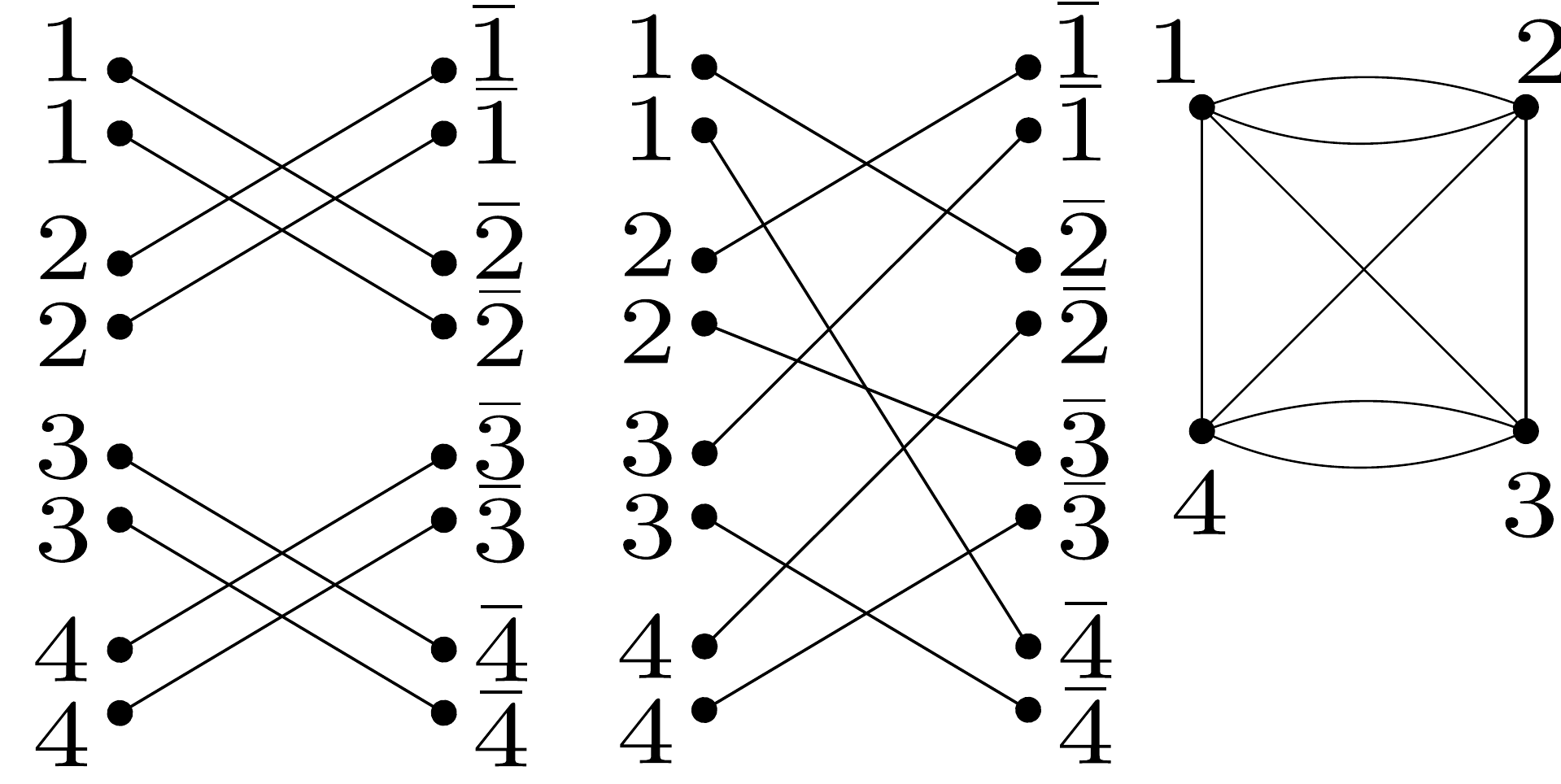}
  \caption{Illustration for $\a=2$ of a regular diagram (left), an irregular diagram $\g$ (centre) and the corresponding reduced diagram $\g^*$ (right).}
  \label{fig: reg and irreg}
\end{figure}
We say that a diagram is regular if the set $\left\{ 1,2,3,4\right\} $
can be partitioned into pairs $\left\{ j,k\right\} $ such that each
edge of the diagram joins a vertex labelled $j$ to $\overline{k}$
or $\overline{j}$ to $k$, otherwise the diagram is said to be irregular; see Figure~\ref{fig: reg and irreg}. Exactly as in \cite{BS}*{Pages 324--5} we have
\begin{equation*}
\left(\frac{1}{2\pi i}\right)^{4} \frac{1}{\alpha^{4}\left(\alpha!\right)^{4}} \sum_{\g\text{ regular}} \int_{\partial D\left(0,r\right)^{4}} \frac{\partial^{4}} {\partial z_{1}\dots\partial z_{4}} v(\g) \prod_{j=1}^{4}\mathrm{d}z_{j} = 3 \, \ex{\big(n_{L}(r;\alpha)\big)^{2}}^{2}
\end{equation*}
and so it is enough to show, for fixed $L\ge \frac{1}{2}$, $\alpha$ and
irregular diagram $\gamma$ that
\begin{equation*}
\int\limits_{\partial D(0,r)^{4}} \frac{\partial^{4}} {\partial z_{1}\dots\partial z_{4}} v(\g) \prod_{j=1}^{4}\mathrm{d}z_{j}  =o\left(\ex{\big(n_{L}(r;\alpha)\big)^{2}}^{2}\right).
\end{equation*}
Combining Lemmas~\ref{lem: 2nd/4th moment chaos} and \ref{lem: bound integral 2nd moment} gives
\begin{equation*}
    \ex{\big(n_{L}(r;\alpha)\big)^{2}}\simeq \frac1{1-r}
\end{equation*}
and so it suffices to show that
\begin{equation*}
\int\limits_{\partial D(0,r)^{4}} \frac{\partial^{4}} {\partial z_{1}\dots\partial z_{4}} v(\g) \prod_{j=1}^{4}\mathrm{d}z_{j} =
o\left(\left(1-r\right)^{-2}\right).
\end{equation*}

We have $v\left(\gamma\right)=\prod_{e}\widehat{K}_{L}\left(e\right)$
and we compute the logarithmic derivative of $v\left(\gamma\right)$
w.r.t. a fixed $z_{j}$. We get
\begin{equation*}
    \frac{\frac{\partial}{\partial z_{j}}v\left(\gamma\right)}{v(\g)}=\sum_e \frac{\frac{\partial}{\partial z_{j}}\widehat{K}_{L}\left(e\right)}{\widehat{K}_{L}\left(e\right)}
\end{equation*}
and note that $\frac{\partial}{\partial z_{j}}\widehat{K}_{L}\left(e\right)$ vanishes unless the edge $e$ joins a vertex labelled $j$ to a vertex labelled $\overline{k}$ for some $k$, or a vertex labelled $\bar{j}$ to a vertex labelled $k^\prime$ for some $k^\prime$. In the former case, using the explicit expression for $\widehat{K}_{L}\left(e\right)$ and differentiating, we get
\[
\frac{\frac{\partial}{\partial z_{j}}\widehat{K}_{L}\left(e\right)}{\widehat{K}_{L}\left(e\right)}= -\frac L2 \frac{\overline{z}_j}{1-|z_j|^2} + L \frac{\overline{z}_k}{1-z_j \overline{z}_k}
\]
while in the latter case we have
\[
\frac{\frac{\partial}{\partial z_{j}}\widehat{K}_{L}\left(e\right)}{\widehat{K}_{L}\left(e\right)}= -\frac L2 \frac{\overline{z}_j}{1-|z_j|^2}.
\]
Since the total number of each type of edge is the same (and equal to $\a$) we get
\[
\frac{\partial}{\partial z_{j}}v\left(\gamma\right)= \frac{L}{1-r^{2}}v\left(\gamma\right) \left( \sum_{k\ne j}E_{j,k}\frac{\overline{z}_{k}-\overline{z}_{j}}{1-z_{j}\overline{z}_{k}}\right)
\]
where $E_{j,k}$ denotes the number of edges joining $j$ to $\overline{k}$, and we have used the fact that $|z_j|=r$. Iterating this, and using the trivial bound $\left|\frac{z-w}{1-z\overline{w}}\right|\le1$,
we see that we can bound
\[
\left|\frac{\partial^{4}}{\partial z_{1}\dots\partial z_{4}}v\left(\gamma\right)\right|\le C\left(L,\gamma\right)\frac{\left|v\left(\gamma\right)\right|}{\left(1-r\right)^{4}}
\]
and so it suffices to see that
\[
\int_{\partial D\left(0,r\right)^{4}} \left|v\left(\gamma\right)\right| \left|\mathrm{d}z_{1}\right| \dots\left|\mathrm{d}z_{4}\right| = o\left(\left(1-r\right)^{2}\right).
\]

We now form a reduced diagram $\gamma^{*}$ by `gluing' together all
of the vertices labelled $j$ or $\overline{j},$ for each $1\le j\le4$; again see Figure~\ref{fig: reg and irreg}. The edges of the resulting diagram have multiplicities, and it is not difficult to see that they must be arranged as
\begin{figure}
  \centering
  \includegraphics[trim={0 15cm 0 0},clip,width=0.4\columnwidth]{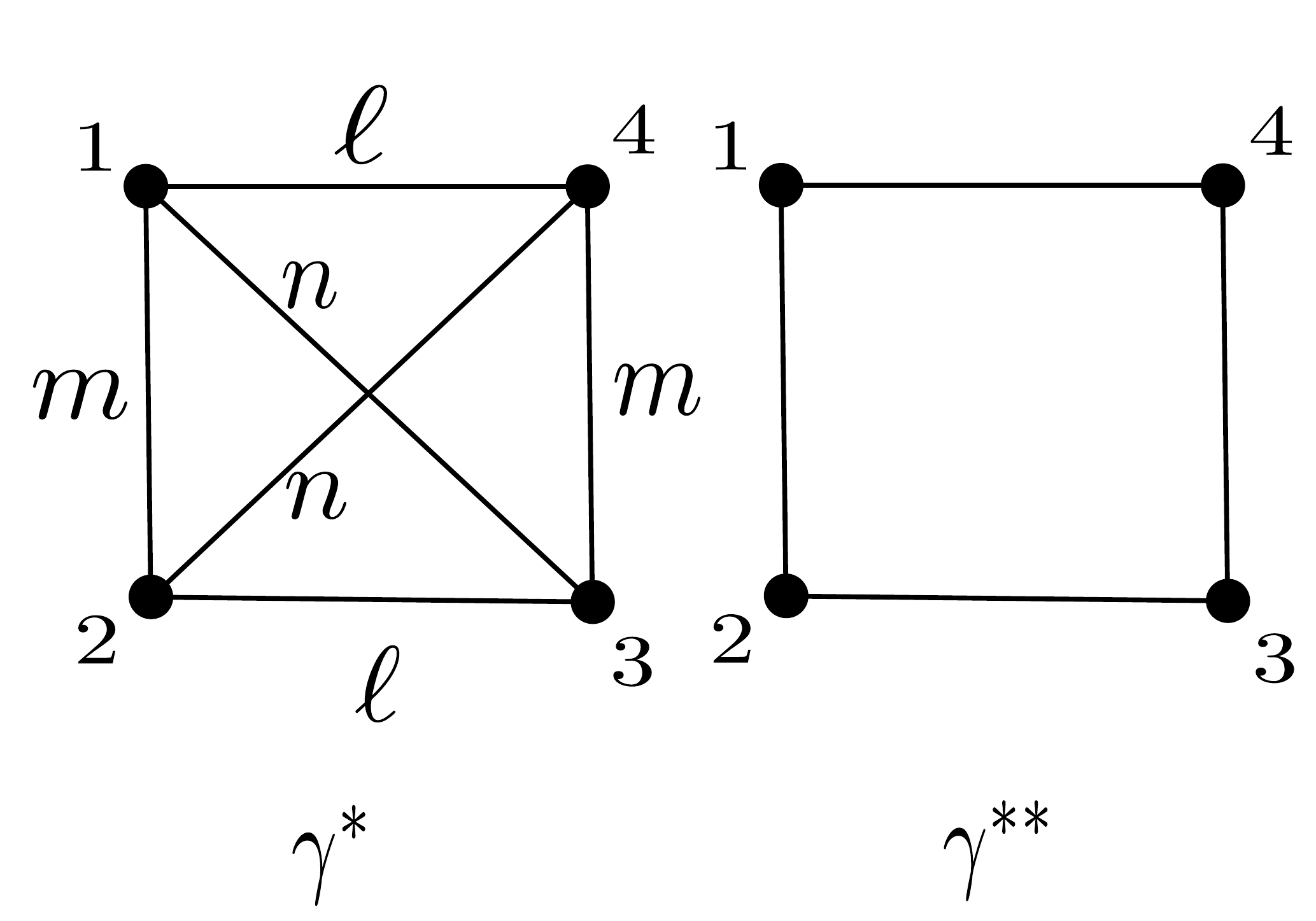}
  \caption{Illustration of a reduced diagram $\g^*$ and the corresponding $\g^{**}$}
  \label{fig: gamma* and gamma**}
\end{figure}
shown on the left of Figure~\ref{fig: gamma* and gamma**}; $\ell$, $m$
and $n$ denote the multiplicity of the edges which satisfy $0\le\ell,m,n<2\alpha$
and $\ell+m+n=2\alpha$. Writing $\mathcal{E}$ for the set of edges we have
\[
\left|v\left(\gamma\right)\right|\le\prod_{\left(j,k\right)\in\mathcal{E}}\left|\widehat{K}_{L}\left(z_{j},z_{k}\right)\right|.
\]
The fact that $\g$ is irregular implies that at most one of $\ell,m,n$ is zero. Now since $\left|\widehat{K}_{L}\left(z,w\right)\right|\le1$ we may delete some of the edges of $\gamma^{*}$ (and re-label the vertices if necessary) to get to the diagram $\gamma^{**}$ depicted on the right of Figure~\ref{fig: gamma* and gamma**}, where each edge has multiplicity $1$. We therefore need to estimate
\begin{align*}
J_{L,r}=\int_{\partial D\left(0,r\right)^{4}}  \Big|\widehat{K}_{L}&\left(z_{1},z_{2}\right)  \widehat{K}_{L}\left(z_{2},z_{3}\right) \widehat{K}_{L}\left(z_{3},z_{4}\right) \widehat{K}_{L}\left(z_{4},z_{1}\right)\Big| \left|\mathrm{d}z_{1}\right| \dots \left|\mathrm{d}z_{4}\right| \\
 & = r^{4}\int_{[-\pi,\pi]^{4}}\left|\widehat{K}_{L}\left(\theta_{2}-\theta_{1}\right)\widehat{K}_{L}\left(\theta_{3}-\theta_{2}\right)\widehat{K}_{L}\left(\theta_{4}-\theta_{3}\right)\widehat{K}_{L}\left(\theta_{1}-\theta_{4}\right)\right|\mathrm{d}\theta_{1}\dots\mathrm{d}\theta_{4},
\end{align*}
where $\widehat{K}_{L}\left(\theta\right)=\widehat{K}_{L,r}\left(\theta\right)=\widehat{K}_{L}\left(r,re^{i\theta}\right)$.
\begin{clm}\label{claim: integral}
Define $I_{L,r}(\t)=1-r$ for $|\t|\le 1-r$ and
\begin{equation*}
     I_{L,r}(\t)=
    \begin{cases}
    \frac{(1-r)^{1+L}}{|\t|^{L}}, & \textrm{if }L>1,\\
    \frac{(1-r)^{2}} {|\t|}\left(1+\log\frac{|\t|}{1-r}\right), & \textrm{if }L=1,\\
    (1-r)^{2L}|\t|^{1-2L}, & \textrm{if }\frac{1}{2}<L<1,
    \end{cases}
\end{equation*}
for $1-r\le|\t|\le\pi$. Extend $I_{L,r}$ to be a $2\pi$-periodic function on $\R$. Then
\begin{equation*}
    \integlim{-\pi}{\pi} {|\widehat{K}_{L}(\t_{4}-\t_{3}) \widehat{K}_{L}(\t_{1}-\t_{4})|} {\t_4} \simeq I_{L,r}(\theta_{3}-\theta_{1}).
\end{equation*}
\end{clm}

By the claim we need to estimate
\[
J_{L,r}\simeq \iint\limits_{[-\pi,\pi]^{2}} I_{L,r}\left(\theta_{3}-\theta_{1}\right)^{2} \mathrm{d}\theta_{1} \mathrm{d}\theta_{3} =2\pi \integlim {-\pi} {\pi} {I_{L,r}(\t)^{2}}{\t}
\]
and doing the integration we get

\[
J_{L,r}\simeq
\begin{cases}
\left(1-r\right)^{3}, & \textrm{if }L>\frac{3}{4},\\
\left(1-r\right)^{3}\log\frac{1}{1-r}, & \textrm{if }L=\frac{3}{4},\\
\left(1-r\right)^{4L}, & \textrm{if }\frac{1}{2}< L<\frac{3}{4},
\end{cases}
\]
so that $J_{L,r}=o\left(\left(1-r\right)^{2}\right)$. It remains only to prove the claim.
\begin{proof}[Proof of Claim~\ref{claim: integral}]
First notice that
\[
\left|\widehat{K}_{L}\left(r,re^{i\theta}\right)\right|=\left(\frac{\left(1-r^{2}\right)^{2}}{\left(1-r^{2}\right)^{2}+2r^{2}\left(1-\cos\theta\right)}\right)^{\frac{L}{2}}\simeq\begin{cases}
1, & \textrm{if }\left|\theta\right|\le\left(1-r\right),\\
\left(\frac{1-r}{\left|\theta\right|}\right)^{L}, & \textrm{if }\left(1-r\right)\le\left|\theta\right|\le\pi.
\end{cases}
\]
Now we just need to do some tedious integration. By periodicity, we
may assume that $\theta_{3}=0$. If $\left|\theta_{1}\right|\ge1-r$
then
\begin{align*}
\integlim{-\pi}{\pi} {|\widehat{K}_{L}(\t_{4}) \widehat{K}_{L}(\t_{1}-\t_{4})|} {\t_4}&\simeq \integlim{|\t_{4}|\le (1-r)/2} {} {\left( \frac{1-r} {|\t_{1}-\t_{4}|} \right)^{L}} {\t_{4}} +\integlim{|\t_{4}-\t_{1}|\le (1-r)/2} {} {\left(\frac{1-r}{|\t_{4}|}\right)^{L}}{\t_{4}}\\
 & \quad +\integlim{|\t_{4}|,|\t_{4}-\t_{1}|\ge (1-r)/2} {} {\left(\frac{\left(1-r\right)^{2}}{\left|\theta_{4}\right|\left|\theta_{1}-\theta_{4}\right|}\right)^{L}}{\theta_{4}}\\
& \simeq 2\cdot\frac{\left(1-r\right)^{1+L}}{\left|\theta_{1}\right|^{L}} + \left(\frac{\left(1-r\right)^{2}} {\left|\theta_{1}\right|}\right)^{L} \integlim {\left(1-r\right)/2\le\left|\theta_{4}\right|\le\left|\theta_{1}\right|/2} {} {\frac{1}{\left|\theta_{4}\right|^{L}}} {\theta_{4}}\\
 & \quad +\left( \frac{\left(1-r\right)^{2}} {\left|\theta_{1}\right|} \right)^{L} \integlim{(1-r)/2\le |\t_{4}-\t_{1}|\le |\t_{1}|/2} {} {\frac{1}{|\theta_{4}-\theta_{1}|^{L} }}{\t_{4}}\\
 & \quad+\integlim {\left|\theta_{4}\right|,\left|\theta_{4}-\theta_{1}\right| \ge \left|\theta_{1}\right|/2} {} {\left(\frac{1-r}{\left|\theta_{4}\right|}\right)^{2L}} {\theta_{4}}
\end{align*}
Performing the integrals we get
\begin{equation*}
    \integlim{-\pi}{\pi} {|\widehat{K}_{L}(\t_{4}) \widehat{K}_{L}(\t_{1}-\t_{4})|} {\t_4} \simeq 
    \begin{cases}
    \frac{\left(1-r\right)^{1+L}}{\left|\theta_{1}\right|^{L}}, & L>1,\\
    \frac{\left(1-r\right)^{2}}{\left|\theta_{1}\right|}\left(1+\log\frac{\left|\theta_{1}\right|}{1-r}\right), & L=1,\\
    \left(1-r\right)^{2L}\left|\theta_{1}\right|^{1-2L}, & \frac{1}{2}<L<1,
\end{cases}
\end{equation*}
for $|\t_1|\ge 1-r$. On the other hand, when $\left|\theta_{1}\right|\le1-r$ we have
\begin{align*}
\integlim{-\pi}{\pi} {|\widehat{K}_{L}(\t_{4}) \widehat{K}_{L}(\t_{1}-\t_{4})|} {\t_4} &
\simeq \integlim {|\t_{4}|,|\t_{4}-\t_{1}|\le 2(1-r)} {} {} {\t_{4}}\\
 & \quad 
 +\integlim {|\t_{4}|\ge2(1-r)\text{ or }|\t_{4}-\t_{1}|\ge2(1-r)} {} {|\widehat{K}_{L}(\t_{4}) \widehat{K}_{L}(\t_{1}-\t_{4})|}  {\theta_{4}}
\end{align*}
Now note that if $|\t_{4}-\t_{1}|\ge2(1-r)$ and $|\t_{1}|\le(1-r)$ then $|\t_{4}|\ge(1-r)$. We get
\begin{equation*}
    \integlim{-\pi}{\pi} {|\widehat{K}_{L}(\t_{4}) \widehat{K}_{L}(\t_{1}-\t_{4})|} {\t_4} \simeq 1-r + \integlim{|\t_{4}|\ge(1-r)}{}{\left( \frac{1-r} {\left|\theta_{4}\right|}\right)^{2L} } {\t_4}\simeq 1-r.\qedhere
\end{equation*}
\end{proof}

\section{Variance for \texorpdfstring{$L=0$}{L=0}}
\label{sec: var comput}
In order to implement the strategy we outlined in Section~\ref{sec: overview} we need sharp estimates for the asymptotic growth of $\var{n_L(r)}$. For $L>0$ these were computed in \cite{B}, and the corresponding result for $L=0$ is as follows.
\begin{prop}\label{prop: var L=0}
We have
\begin{equation*}
    \var{n_{0}\left(r\right)}\overset{r\to1}{\sim} \frac{\pi^{2}} {24\left(1-r\right)^{2} \big(\log\frac{1}{1-r}\big)^{4}}.
\end{equation*}

\end{prop}

We will actually give a proof that recovers the result from \cite{B} for $L<\tfrac12$ (with no regard for the error term). Consider a GAF (in the unit disc) of the form
\[
f(z) = \sum_{m\ge 0} b_m \zeta_m z^m
\]
where $b_m\ge 0$ and denote by
\[
G(z) = \sum_{m \ge 0} b_m^2 z^m
\]
its \emph{covariance function}, so that $K(z,w) =\ex{f(z)\overline{f(w)}}= G(z \bar{w})$ is the corresponding covariance kernel. It will be convenient to use the following notation from \cite{KN} for $\Re{\z} \le 0$:
\[
H(\z) = G\big(e^\z\big),\qquad A(\z) = \frac{H^\prime(\z)}{H(\z)}.
\]
Using the above notation we have the following formula for the variance of the number of zeroes of $f$ in the disc of radius $r<1$ (see \cite{KN}*{Appendix A})
\begin{equation}\label{eq: varformula}
\var{n_{f}(r)} = \frac{1}{2\pi} \integlim{-\pi}{\pi}{\frac{|H(t)H^\prime(t + i \theta) - H(t + i \theta) H^\prime(t)|^2}{H^2(t) (H^2(t) - |H(t + i \theta)|^2)}}{\theta}, \quad \mbox{ with } e^t = r^2.
\end{equation}

We write $\z = t + i \theta$, with $t \le 0$, and mention that $|H(t+i\theta)| \le H(t)$. In order to derive an asymptotic expression for the variance we will make the following assumptions; they will allow us to show that the integrand above may be approximated by $\frac{A^2(t)} {H^2(t)} |H(i\theta)|^2$ when $\t$ does not belong to a small neighbourhood of $0$ and apply the Dominated Convergence Theorem.
\begin{assume}
    \item \label{l2cond} $\sum_{m\ge 0} b_m^4 < \infty.$

    \item \label{growthH} $H(t + i \theta) = o(H(t))$ as $t \to 0^{-}$ for every $\theta \in [-\pi,\pi] \setminus \{0\}$.
    
    \item \label{growthA} There exist $t_0<0$ and a constant $C 
    \ge 1$ such that for $t_0<t<0$ we have $\left| A(t + i \theta) \right| \le C A(t)$ for every $\t\in[-\pi,\pi]$. Furthermore, for every $\theta \in [-\pi,\pi] \setminus \{0\}$ we have $A(t+i\theta) = o\left(A(t)\right)$ as $t \to 0^{-}$.
    
    \item \label{boundH} There is a function $\Delta: (-\infty, 0) \to [0,\pi]$ such that $\Delta(t) \downarrow 0$, as $t\to 0^{-}$ and moreover if $t$ is sufficiently close to $0$ then
    \begin{equation*}
        |\theta| \ge \Delta(t)\implies |H(t + i \theta)| \le \tfrac12 H(t).
    \end{equation*}
    
    \item \label{boundVar} Additionally $\Delta(t) A^\prime(t) = o(A^2(t) H^{-2}(t))$ as $t\to 0^{-}$.
\end{assume}

\begin{rmk}
Assumption \eqref{l2cond} implies that $G$ belongs to the Hardy space $H^2(\D)$. Thus, there is a function $M \in L^2([-\pi,\pi])$ (e.g., the radial or non-tangential maximal function) that satisfies, for $\theta \in [-\pi,\pi]$,
\begin{equation}\label{Hmajorant}
\sup_{t < 0} |H(t + i\theta)| \le M(\theta).
\end{equation}
In particular, the (radial) boundary values $H(i \theta)$ exist for a.e. $\theta \in \T$.

\end{rmk}

\begin{prop}\label{varasymp}
Put $e^t = r^2$ and let $f$ be a GAF whose covariance function $G$ satisfies the above assumptions. Then
\[
\var{n_{f}(r)} = (1+o(1)) \frac{A^2(t)}{H^2(t)} \sum_{m\ge0} b_m^4 = (1+o(1)) \frac{(G^\prime(r^2))^2}{(G(r^2))^4} \sum_{m\ge0} b_m^4, \mbox{ as } r \to 1.
\]
\end{prop}

In our case we have
\[
H_L(\z) =
\begin{cases}
-\log(1 - e^\z) & ,\,L = 0;\\
(1-e^\z)^{-L} & ,\,L \in (0, \tfrac12),\\
\end{cases}\quad
A_L(\z) =
\begin{cases}
-\frac{e^\z}{\left(1-e^\z\right) \log \left(1-e^\z\right)} & ,\,L = 0;\\
\frac{L e^\z}{1-e^\z} & ,\,L \in (0, \tfrac12),\\
\end{cases}
\]
and
\[
A_L^\prime(\z) =
\begin{cases}
-\frac{e^\z \left(e^\z+\log \left(1-e^\z\right)\right)}{\left(1-e^\z\right)^2 \log ^2\left(1-e^\z\right)} & ,\,L = 0;\\
\frac{L e^\z}{\left(1 - e^\z\right)^2} & ,\,L \in (0, \tfrac12).\\
\end{cases}
\]
It is not difficult to check that the GAFs $f_L$ satisfy the assumptions of Proposition \ref{varasymp}, for $L \in [0,\tfrac12)$, where we may take
\[
\Delta_L(t) =
\begin{cases}
\frac{1}{\log^4\left(1-e^t\right)} & ,\,L = 0;\\
\left(1 - e^t\right)^{L + \tfrac12} & ,\,L \in (0, \tfrac12).\\
\end{cases}
\]
We conclude that, as $r \to 1$,
\[
\var{n_{L}(r)}\sim \begin{cases}
\frac{\pi^2}{24} \cdot \frac{1}{(1-r)^2 \log^4(1-r)} & ,\,L = 0;\\
\frac{L^2 \Gamma(1-2L)}{4^{1-L}\Gamma^2(1-L)} \cdot \frac{1}{(1-r)^{2(1-L)}} & ,\,L \in (0, \tfrac12).\\
\end{cases}
\]

\begin{proof}[Proof of Proposition \ref{varasymp}]
We rewrite the integrand in \eqref{eq: varformula} as follows
\[
I_H(t;\theta) = \frac{|H(t)H^\prime(t + i \theta) - H(t + i \theta) H^\prime(t)|^2}{H^2(t) (H^2(t) - |H(t + i \theta)|^2)} = \frac{|A(t) - A(t + i \theta)|^2 |H(t+i\theta)|^2}{H^2(t) - |H(t + i \theta)|^2},
\]
and split the integral
\[
\var{n_{f}(r)} = \frac{1}{2\pi} \integllim{J_1 \uplus J_2}{I_H(t;\theta)}{\theta},
\]
where $J_1 = \{ \Delta(t) \le |\theta| \le \pi \}$ and $J_2 = \{ |\theta| \le \Delta(t) \}$. 

On $J_1$ we re-write the integrand as
\[
I_H(t;\theta) = \frac{A^2(t) |H(t + i\theta)|^2}{H^2(t)} R(t;\theta).
\]
where
\[
R(t;\theta) = \left| 1 - \frac{A(t+i\theta)}{A(t)}\right|^2 \left(1 - \frac{|H(t + i\theta)|^2}{H^2(t)} \right)^{-1}.
\]
By Assumptions \eqref{growthA} and \eqref{boundH} we have that $|R(t;\theta)| \le C^\prime$ for $t$ sufficiently close to $0$. Combining this with \eqref{Hmajorant} we see that we may apply the Dominated Convergence Theorem, and using Assumptions \eqref{growthH} and \eqref{growthA} we get
\[
\lim_{t\to0^{-}} \frac{1}{2\pi} \integllim{J_1}{I_H(t;\theta) \left(\frac{A^2(t)}{H^2(t)} \right)^{-1}}{\theta}
= \frac{1}{2\pi} \integlim{-\pi}{\pi}{|H(i \theta)|^2}{\theta} = \sum_{m\ge 0} b_m^4.
\]

On $J_2$, we use the bound\footnote{\cite{KN} uses the notation $B = A^\prime$.} $I_H(t;\theta) \le A^\prime(t)$ which is valid for all $\theta \in [-\pi,\pi]$ and $t<0$ (\cite{KN}*{Corollary 5.3}). Using Assumption \eqref{boundVar} we get
\[
\frac{1}{2\pi} \integllim{J_2}{I_H(t;\theta)}{\theta} \le \frac{\Delta(t)}{\pi} A^\prime(t) = o\left( \frac{A^2(t)}{H^2(t)} \right), \quad \mbox{ as } t \to 0^{-}.\qedhere
\]
\end{proof}

\section{Non-CLT for \texorpdfstring{$0\le L< \tfrac12$}{0<L<1/2} and CLT for  \texorpdfstring{$L=\tfrac12$}{L=1/2}}\label{sec: L<1/2}
In this section we complete the proof of Theorem~\ref{thm: main}. As we outlined in Section~\ref{sec: overview}, we will show that the main contribution comes from $n_{L}(r;1)$ and we begin by deriving another expression for it.

\subsection{An explicit formula for \texorpdfstring{$n_{L}(r;1)$}{nL(r;1)}.}

We begin with an elementary but useful lemma about deterministic power series.
\begin{lem}
Suppose that $f(z)=\sum b_{m}z^{m}$ has radius of convergence $1$. Then, for $0<r<1$,
\begin{itemize}
\item $\frac{1}{2\pi i}\integlim {\partial D\left(0,r\right)} {} {\overline{z}\left|f(z)\right|^{2}} {z} = r^{2} \sum |b_{m}|^{2}r^{2m}$, and
\item $\frac{1}{2\pi i} \integlim {\partial D\left(0,r\right)} {} {f'(z)\overline{f(z)}} {z}= \sum m |b_{m}|^{2}r^{2m}$.
\end{itemize}

\end{lem}

\begin{proof}
We have
\begin{align*}
\frac{1}{2\pi i} \integlim {\partial D\left(0,r\right)} {} {\overline{z}\left|f(z)\right|^{2}} {z} & = \frac{1}{2\pi i} \integlim {-\pi} {\pi} {re^{-i\theta} \sum_{m,m'=0}^{\infty}b_{m}\overline{b}_{m'}r^{m+m'} e^{i\theta\left(m-m'\right)} ri e^{i\theta}} {\theta}\\
& =r^{2} \sum_{m,m'=0}^{\infty}b_{m}\overline{b}_{m'}r^{m+m'} \frac{1}{2\pi} \integlim {-\pi} {\pi} {e^{i\theta\left(m-m'\right)}} {\theta}\\
& =r^{2} \sum_{m=0}^{\infty}\left|b_{m}\right|^{2}r^{2m}
\end{align*}
since the Taylor series that defines $f$ converges uniformly on compact
subsets of $\D$. Similarly
\begin{align*}
\frac{1}{2\pi i} \integlim {\partial D\left(0,r\right)} {} {f'(z)\overline{f(z)}} {z} & =\frac{1}{2\pi i} \integlim {-\pi} {\pi} {\sum_{m,m'=0}^{\infty} m b_{m}\overline{b}_{m'} r^{m-1+m'} e^{i\t(m-1-m')} ri e^{i\t}} {\t}\\
& =\sum_{m,m'=0}^{\infty} m b_{m}\overline{b}_{m'} r^{m+m'} \frac{1}{2\pi} \integlim {-\pi} {\pi} {e^{i\t (m-m')}} {\t}\\
 & =\sum_{m=0}^{\infty} m |b_{m}|^{2} r^{2m}.
\end{align*}
\end{proof}
\begin{prop}\label{prop: first chaos}
We have
\[
n_{L}\left(r;1\right)=\left(1-r^{2}\right)^{L-1} \sum_{m=0}^{\infty}a_{m,L}\big(m\left(1-r^{2}\right)-Lr^{2}\big)\left(\left|\zeta_{m}\right|^{2}-1\right)r^{2m}
\]
for $L>0$ and
\[
n_{0}\left(r;1\right)=\frac{1}{\left(1-r^{2}\right) \left(\log\frac{1}{1-r^{2}}\right)^2}\sum_{m=1}^{\infty}a_{m,0}\left(m\left(1-r^{2}\right)\log\frac{1}{1-r^{2}}-r^{2}\right)\left(\left|\zeta_{m}\right|^{2}-1\right)r^{2m}.
\]
\end{prop}
\begin{proof}
Recall that
\[
n_{L}\left(r;\alpha\right)= \frac{\left(-1\right)^{\alpha+1}} {\alpha\left(\alpha!\right)} \frac{1}{2\pi i} \integlim {\partial D\left(0,r\right)} {} {\frac{\partial}{\partial z} \wikk{\widehat{f}_{L}\left(z\right)} {\alpha}} {z}
\]
and that $\wikk{\zeta}{}=\left|\zeta\right|^{2}-1$. This yields,
for $L>0$,
\begin{align*}
n_{L}\left(r;1\right) & =\frac{1}{2\pi i} \integlim {\partial D\left(0,r\right)}{}{\frac{\partial}{\partial z}\left|\widehat{f}_{L}\left(z\right)\right|^{2}}{z}\\
& =\frac{1}{2\pi i} \integlim {\partial D\left(0,r\right)} {} {-L\overline{z} \left(1-|z|^{2}\right)^{L-1} \left|f_{L}\left(z\right)\right|^{2} + \left(1-|z|^{2}\right)^{L}f'_{L}\left(z\right)\overline{f_{L}\left(z\right)}}{z}\\
& =\left(1-r^{2}\right)^{L-1} \sum_{m=0}^{\infty}a_{m,L} \big(m\left(1-r^{2}\right) -Lr^{2}\big)\left| \zeta_{m}\right|^{2}r^{2m}
\end{align*}
where the last equality follows from the previous lemma. Since\footnote{It is also possible to verify that $\sum_{m=0}^{\infty} a_{m,L} \big(m (1-r^{2}) -Lr^{2}\big) r^{2m}=0$ directly.} $\ex{n_{L}\left(r;1\right)}=0$ and $\ex{|\z_m|^2}=1$ the result for $L>0$ follows. The case $L=0$ is similar and omitted.
\end{proof}

\subsection{The second chaos dominates}
In this section we show that the main contribution to $n_L(r)$ comes from $n_L(r;1)$, the projection to the second chaos. (Recall that the odd chaoses vanish and $n_L(r;\a)$ denotes the $2\a$-th chaos.)
\begin{prop}\label{prop: var first chaos}
If $0\le L\le\frac{1}{2}$ is fixed then
\[
\ex{n_{L}\left(r;1\right)^{2}} \sim\var{n_{L}\left(r\right)}
\]
as $r\to1$. Moreover, this asymptotic also holds as $L\to\frac{1}{2}$, $r\to1$ in an arbitrary way.
\end{prop}
\begin{rmks}\phantom{w}
\begin{enumerate}
    \item From \cite{B} we have
    \begin{equation}
    \var{n_{L}\left(r\right)}\sim
    \begin{cases}
    \frac{L^{2}\Gamma\left(\frac{1}{2}-L\right)} {4\sqrt{\pi}\Gamma\left(1-L\right)} \left(1-r\right)^{2L-2}, & 0<L<\frac{1}{2};\\
    \frac{1}{8\pi\left(1-2L\right)} \left(1-r\right)^{2L-2} \big(1-\left(1-x\right)^{1-2L}\big), & L\to\frac{1}{2},\, L \ne \frac12 ;\\
    \frac{1}{8\pi} \frac{1}{1-r} \log\frac{1}{1-r}, & L=\frac{1}{2}.
    \end{cases}\label{eq: asym for var}
    \end{equation}
    The case $L=0$ is given in Proposition~\ref{prop: var L=0}.
    
    \item Indeed it is possible to mimic the proofs given in \cite{B} to prove Proposition~\ref{prop: var first chaos}, but we give a different one for variety. Moreover this proof will boil down to proving estimates that will be necessary for the proof of Theorem~\ref{thm: main}.
\end{enumerate}
\end{rmks}
We begin with a useful lemma. We put
\begin{equation*}
    \psi_{L}\left(x\right)=
    \begin{cases} \frac{1-\left(1-x\right)^{1-2L}}{\pi\left(1-2L\right)}, & \textrm{for }L\neq\frac{1}{2};\\
    \frac{1}{\pi}\log\frac{1}{1-x}, & \textrm{for }L=\frac{1}{2}.
    \end{cases}
\end{equation*}
\begin{lem}
\label{lem: hypgeom asym}If $0<L<\frac{1}{2}$ is fixed then
\begin{equation}
\sum_{m=0}^{\infty}a_{m,L}^{2}x^{m}\xrightarrow{x\to1^-}\frac{\Gamma\left(1-2L\right)}{\Gamma\left(1-L\right)^{2}}=\frac{\Gamma\left(\frac{1}{2}-L\right)}{4^{L}\sqrt{\pi}\Gamma\left(1-L\right)}.\label{eq:hyp fixed}
\end{equation}
If $x\to1^-$ and $L\to\frac{1}{2}$ then
\[
\sum_{m=0}^{\infty}a_{m,L}^{2}x^{m}\sim\psi_{L}\left(x\right).
\]

\end{lem}
\begin{proof}
For $0<L<\frac{1}{2}$ and $0<x<1$ we have\footnote{In fact this is Gauss's integral representation for the hypergeometric function $_{2}F_{1}\left[L,L;1;x\right]$}
\[
\sum_{m=0}^{\infty} a_{m,L}^{2}x^{m} = \frac{1}{\Gamma\left(L\right)\Gamma\left(1-L\right)} \integlim {0} {1} {t^{L-1}\left(1-t\right)^{-L}\left(1-xt\right)^{-L}} {t},
\]
which is easily verified by expanding the term $\left(1-xt\right)^{-L}$
as a power series in $xt$. Since both sides are convergent when $x=1$
we have
\[
\sum_{m=0}^{\infty} a_{m,L}^{2} x^{m} \overset{x\to1}{\sim} \sum_{m=0}^{\infty} a_{m,L}^{2} = \frac{1}{\Gamma\left(L\right)\Gamma\left(1-L\right)} \integlim {0} {1} {t^{L-1}\left(1-t\right)^{-2L}} {t} = \frac{B\left(L,1-2L\right)} {\Gamma\left(L\right)\Gamma\left(1-L\right)} = \frac{\Gamma\left(1-2L\right)}{\Gamma\left(1-L\right)^{2}}
\]
where $B$ is the beta function. Applying the identity
\[
\Gamma\left(z\right)\Gamma\left(z+\frac{1}{2}\right)=2^{1-2z}\sqrt{\pi}\Gamma\left(2z\right)
\]
yields \eqref{eq:hyp fixed}.

Now suppose that $x\to1^-$ and $L\to\frac{1}{2}^-$ simultaneously and write
$\e=\e \left(x\right)=\left(1-x\right)\log\frac{1}{1-x}$. For $0<t<1-\e$
we have
\[
\left(1-xt\right)^{-L}=\left(1-t\right)^{-L}\left(1+\frac{\left(1-x\right)t}{1-t}\right)^{-L}\sim\left(1-t\right)^{-L}
\]
uniformly, which yields
\begin{align*}
\integlim {0} {1-\epsilon} {t^{L-1} (1-t)^{-L} (1-xt)^{-L}} {t} & \sim \integlim {0} {1-\e} {t^{L-1} (1-t)^{-2L}} {t}\\
& =B\left(L,1-2L\right) - \integlim {1-\e} {1} {t^{L-1} (1-t)^{-2L}} {t}\\
& \sim \frac{1-\e^{1-2L}} {1-2L} \sim \frac{1-(1-x)^{1-2L}} {1-2L}.
\end{align*}
Further
\begin{align*}
\integlim {1-\e} {1} {t^{L-1} (1-t)^{-L} \left(1-xt\right)^{-L}} {t} & \le (1-x)^{-L} \integlim{1-\e} {1} {t^{L-1}\left(1-t\right)^{-L}} {t}\\
& \lesssim \left(1-x\right)^{-L} \integlim {0} {\epsilon} {s^{-L}} {s}\\
& =\left(1-x\right)^{1-2L}\left(\log\frac{1}{1-x}\right)^{1-L}=o\left(\frac{1-\left(1-x\right)^{1-2L}}{1-2L}\right).
\end{align*}
Since $\Gamma\left(\frac{1}{2}\right)=\sqrt{\pi}$ we have $\Gamma\left(L\right)\Gamma\left(1-L\right)\to\pi$ and the result follows in this case.

For $L\ge\frac{1}{2}$ we use the asymptotic (which follows from Stirling's bounds)
\[
\frac{\Gamma\left(L+m\right)}{m!}=m^{L-1}+O\left(m^{L-2}\right)
\]
for $m\ge1$, and the implicit constant is uniform for $L$ bounded away from $0$ and $\infty$. For $L=\frac{1}{2}$ we get
\[
\sum_{m=0}^{\infty}a_{m,L}^{2}x^{m}=1+\frac{1}{\pi}\sum_{m=1}^{\infty}\left(\frac{1}{m}+O\left(\frac{1}{m^{2}}\right)\right)x^{m}=\frac{1}{\pi}\log\frac{1}{1-x}+O\left(1\right).
\]
For $L\to\frac{1}{2}^+$ we see that
\[
\sum_{m=0}^{\infty}a_{m,L}^{2}x^{m}=1+\frac{1}{\Gamma(L)^{2}}\sum_{m=1}^{\infty}\left(\frac{1}{m^{2-2L}}+O\left(\frac{1}{m^{3-2L}}\right)\right)x^{m} \sim \frac{1}{\pi}\sum_{m=1}^{\infty}\frac{x^{m}}{m^{2-2L}}.
\]
On the other hand
\[
\frac{1-\left(1-x\right)^{1-2L}}{1-2L}=\frac{1}{\left(2L-1\right)\Gamma\left(2L-1\right)}\sum_{m=1}^{\infty}\frac{\Gamma\left(2L-1+m\right)}{m!}x^{m}\sim\sum_{m=1}^{\infty}\frac{x^{m}}{m^{2-2L}}.
\]
\end{proof}

\begin{proof}[Proof of Proposition~\ref{prop: var first chaos}]
Since the random variables $\left|\zeta_{m}\right|^{2}-1$ are orthonormal we get, for $L>0$,
\[
\ex{n_{L}\left(r;1\right)^2}= \left(1-r^{2}\right)^{2L-2} \sum_{m=0}^{\infty} a_{m,L}^{2} \big(m\left(1-r^{2}\right)-Lr^{2}\big)^{2}r^{4m}.
\]
We expand  the term $\big(m\left(1-r^{2}\right)-Lr^{2}\big)^{2}$ and estimate, bearing in mind \eqref{eq: asym for var}. From Lemma~\ref{lem: hypgeom asym} we see that
\[
L^{2}r^4 \left(1-r^{2}\right)^{2L-2} \sum_{m=0}^{\infty} a_{m,L}^{2} r^{4m} \sim\var{n_{L}\left(r\right)}
\]
for $0<L\le\frac{1}{2}$ and for $L\to\frac{1}{2}$. We also have, for $L>0$,
\[
(1-r^{2})^{2L} \sum_{m=0}^{\infty} a_{m,L}^{2} m^{2} r^{4m} \simeq (1-r)^{2L} \sum_{m=1}^{\infty}m^{2L}r^{4m} \simeq(1-r)^{-1} 
\]
and this is all uniform when $L$ is close to $\frac12$. Similarly
\[
Lr^{2}(1-r^{2})^{2L-1} \sum_{m=0}^{\infty}a_{m,L}^{2}mr^{4m} \simeq(1-r)^{2L-1} \sum_{m=1}^{\infty} m^{2L-1}r^{4m} \simeq(1-r)^{-1}
\]
and furthermore $(1-r)^{-1}=o\left(\var{n_{L}\left(r\right)}\right)$ if $0<L\le\frac12$ or $L\to\frac12$.

Similarly
\begin{align*}
\ex{n_{0}\left(r;1\right)^{2}} & = \frac{1}{\left(1-r^{2}\right)^{2} \big(\log\frac{1}{1-r}\big)^{4}} \sum_{m=1}^{\infty}\left(\left(1-r^{2}\right)\log\frac{1}{1-r^{2}}-\frac{r^{2}}{m}\right)^{2}r^{4m}\\
&\sim 
\frac{1} {4 \left(1-r\right)^{2} \big(\log\frac{1}{1-r}\big)^{4}} \sum_{m=1}^{\infty} \frac{1}{m^{2}} = \frac{\pi^{2}} {24\left(1-r\right)^{2} \big(\log\frac{1}{1-r}\big)^{4}}
\end{align*}
as claimed.
\end{proof}

\subsection{Completing the proof of Theorem~\ref{thm: main}}

\begin{proof}[Proof of Theorem~\ref{thm: main}~\eqref{part: 0<L<half}]
We begin with the case $0<L<\frac12$. By Proposition~\ref{prop: var first chaos} it is enough to see that $\frac{n_{L}\left(r;1\right)} {\var{n_{L}\left(r\right)} } \to -c_{L}X_{L}$
in $L^2$. Using the estimate \eqref{eq: asym for var}, the alternative expression for $c_L$ given by \eqref{eq:hyp fixed} and Proposition~\ref{prop: first chaos}, we see that it is  enough
to show that
\[
\frac{1}{L}\sum_{m=0}^{\infty} a_{m,L} \left(m\left(1-r^{2}\right)-Lr^{2}\right) \left(\left|\zeta_{m}\right|^{2}-1\right) r^{2m} \xrightarrow{L^2} -X_{L}.
\]
But we have
\[
\ex{\bigg(\sum_{m=0}^{\infty} a_{m,L} m\left(1-r^{2}\right) \left(\left|\zeta_{m}\right|^{2}-1\right) r^{2m}\bigg)^{2} } = \left(1-r^{2}\right)^{2} \sum_{m=0}^{\infty} a_{m,L}^{2} m^{2} r^{4m} \simeq\left(1-r\right)^{1-2L} =o\left(1\right)
\]
as before, and so we need to show that
\[
-r^{2} \sum_{m=0}^{\infty} a_{m,L} \left(\left|\zeta_{m}\right|^{2}-1\right) r^{2m} \xrightarrow{L^2} -X_{L}.
\]
This will follow if we show that
\[
\sum_{m=0}^{\infty} a_{m,L}^{2} \left(r^{2m+2}-1\right)^{2}\to0,
\]
but this is obvious since $\sum_{m=0}^{\infty}a_{m,L}^{2}<+\infty$.

Next we treat the case $L=0$. We need to see that
\[
\sum_{m=1}^{\infty} a_{m,0} \left(m\left(1-r^{2}\right) \log\frac{1}{1-r^{2}}-r^{2}\right) \left(\left|\zeta_{m}\right|^{2}-1\right) r^{2m} \xrightarrow{L^2} -X_{0}.
\]
Similar to before we see that
\begin{align*}
\ex{\left(\sum_{m=1}^{\infty} a_{m,0} m \left(1-r^{2}\right) \log\frac{1}{1-r^{2}} \left(\left|\zeta_{m}\right|^{2}-1\right) r^{2m }\right)^{2} } &=  \left(1-r^{2}\right)^{2} \Big(\log\frac{1}{1-r^{2}}\Big)^{2} \sum_{m=1}^{\infty} r^{4m}\\
&\simeq  \left(1-r\right) \Big(\log\frac{1}{1-r^{2}}\Big)^{2} =o\left(1\right)
\end{align*}
and that
\[
\sum_{m=1}^{\infty} a_{m,0}^{2} \left(r^{2m +2}-1\right)^{2} \to0.\qedhere
\]
\end{proof}
Finally we treat the case $L\to\frac12$, which of course includes the case $L=\frac{1}{2}$, and so completes the proof of Theorem~\ref{thm: main}. We first state a convenient lemma.

\begin{lem}
\label{lem: 4th mom sum}Let $\zeta_{m}$ be a sequence of iid $\mathcal{N}_{\mathbb{C}}(0,1)$
random variables as before and $\alpha_{m}$ be square-summable real
coefficients. Then
\[
S=\sum_{m=0}^{\infty}\alpha_{m}(|\zeta_{m}|^2-1)
\]
is almost surely convergent, has mean $0$, and satisfies
\[
\E\left[S^{2}\right]=\sum_{m=0}^{\infty}\alpha_{m}^{2}\quad\text{and}\quad
\E\left[S^{4}\right]=3\left(\sum_{m=0}^{\infty}\alpha_{m}^{2}\right)^{2}+6\sum_{m=0}^{\infty}\alpha_{m}^{4}.
\]
\end{lem}

\begin{rmk}
The Gaussianity plays no role here, we may replace $|\zeta_{m}|^2-1$ by any iid real-valued random variables with the same mean, variance and fourth moment. 
\end{rmk}

\begin{proof}[Proof of Theorem~\ref{thm: main}~\eqref{part: L to half} and the case $L=\frac12$]
As before, by Proposition~\ref{prop: var first chaos} it is enough
to see that $\frac{n_{L}\left(r;1\right)}{\sqrt{\E\left[n_{L}\left(r;1\right)^{2}\right]}}$
is asymptotically normal. By the Fourth Moment Theorem \cite{fourth mom}*{Theorem 1}
this is equivalent to
\[
\frac{\E\left[n_{L}\left(r;1\right)^{4}\right]}{\E\left[n_{L}\left(r;1\right)^{2}\right]^{2}}\to3.
\]
Combining Lemma~\ref{lem: 4th mom sum}, Proposition~\ref{prop: first chaos}
and the estimate \eqref{eq: asym for var} we see that it is enough to
show that
\[
\sum_{m=0}^{\infty}a_{m,L}^{4}\left(m\left(1-r^{2}\right)-Lr^{2}\right)^{4}r^{8m}=o\left(\psi_{L}(r)^{2}\right).
\]
In fact we will show that
\[
\sum_{m=0}^{\infty}a_{m,L}^{4}\left(m\left(1-r^{2}\right)-Lr^{2}\right)^{4}r^{8m}=O\left(1\right).
\]
First note that for $1\le k\le4$ we have
\[
\left(1-r\right)^{k}\sum_{m=0}^{\infty}a_{m,L}^{4}m^{k}r^{8m}\simeq\left(1-r\right)^{k}\sum_{m=0}^{\infty}m^{4L-4+k}r^{8m}\simeq\left(1-r\right)^{3-4L}=o\left(1\right)
\]
and
\[
L^{4}\sum_{m=0}^{\infty}a_{m,L}^{4}r^{8m}\lesssim L^{4}\sum_{m=0}^{\infty}m^{4L-4}=O\left(1\right)
\]
for $L<\frac{3}{4}$. This completes the proof.
\end{proof}

\begin{proof}[Proof of Lemma~\ref{lem: 4th mom sum}]
To shorten the expressions we write $\wikk{\z}{}=|\z|^2-1$. The fact that $S$ is convergent is an obvious consequence, e.g., of Kolmogorov's Three Series Theorem, and since $\ex{\wikk{\z_{m}}{}}=0$, it must have mean $0$ also. Since the sequence $\wikk{\z_{m}}{}$ is orthogonal in $L^2$ we see that
\[
\E\left[S^{2}\right]= \sum_{m,m'=0}^{\infty} \alpha_{m} \alpha_{m'} \ex{\wikk{\z_{m}}{}\,\wikk{\z_{m'}}{}} =\sum_{m=0}^{\infty}\alpha_{m}^{2}.
\]
Finally note that a straightforward computation gives
\[
\ex{\wikk{\z_{m_1}}{}\,\wikk{\z_{m_2}}{}\,\wikk{\z_{m_3}}{}\, \wikk{\z_{m_4}}{}} = 
\begin{cases}
9, & \text{if }m_{1}=m_{2}=m_{3}=m_{4};\\
1, & \text{if }\{m_{1},m_{2},m_{3},m_{4}\}=\{m,m'\}\text{ where }m\neq m';\\
0, & \text{otherwise.}
\end{cases}
\]
This yields
\begin{align*}
\E\left[S^{4}\right] & =\sum_{m_{1},m_{2},m_{3},m_{4}=0}^{\infty} \alpha_{m_{1}} \alpha_{m_{2}} \alpha_{m_{3}} \alpha_{m_{4}} \ex{\wikk{\z_{m_1}}{}\, \wikk{\z_{m_2}}{}\, \wikk{\z_{m_3}}{}\, \wikk{\z_{m_4}}{}}\\
 & =6\left(\sum_{m=0}^{\infty}\sum_{m'<m}\alpha_{m}^{2}\alpha_{m'}^{2}\right)+9\sum_{m=0}^{\infty}\alpha_{m}^{4} =3\left(\sum_{m=0}^{\infty}\sum_{m'\ne m}\alpha_{m}^{2}\alpha_{m'}^{2}\right)+9\sum_{m=0}^{\infty}\alpha_{m}^{4}\\
 & =3\left(\sum_{m=0}^{\infty}\alpha_{m}^{2}\right)^{2}+6\sum_{m=0}^{\infty}\alpha_{m}^{4}.\qedhere
\end{align*}
\end{proof}

\subsection{A comment on the case \texorpdfstring{$L\to0^+$}{L0}} Theorem~\ref{thm: main} does not cover the case $L\to0^+$. We believe that the behaviour should be as follows, although the computations needed to prove it by our methods appear formidable. For $x\ge0$ we define
\begin{equation*}
    \Phi_{\eta}\left(x\right)=
    \begin{cases}
    1-x^2, & \eta=0,\\
    \frac{\sqrt{1-\eta^{2}}}{\eta}\left(\frac{e^{-\frac{\eta}{1-\eta}x^2}}{1-\eta}-1\right), & 0<\eta<1;\\
    0 & \eta=1.
\end{cases}
\end{equation*}
If $L\to0$, $r\to1$, $\frac{L}{1-r}\to\infty$ and $\left(1-r\right)^{L}\to\eta$ then we should have $\hat{n}_{L}(r)\to\Phi_{\eta}\left(\abs{\zeta_{0}}\right)$ in $L^2$. 

Let us explain the importance of $\eta$. If $m\ge1$ is fixed then $a_{m,L}\sim\frac{L}{m}$ as $L\to0$, which means that, almost surely, $\frac{1}{\sqrt{L}}\left(f_{L}-\zeta_{0}\right)\to f_{0}$ locally uniformly. Heuristically $f_L(z)=0$ corresponds to $f_0(z)=-\z_0/\sqrt L$. The typical size of $|f_0(z)|^2$ is $\log\frac{1}{1-|z|}$ and $\eta$ determines whether
$1/L$ is much smaller than, roughly the same size as, or much larger than this value.

\section{Tail asymptotics}\label{sec: tail}
We will finally prove Theorem~\ref{thm: tail}. Since the distribution of the random variable $\abs{\z_m}^2$ is $\Exp\left(1\right)$ whose characteristic function is $\frac{1}{1-it}$ we can calculate (recall \eqref{eq: amL def})
\begin{equation*}
\varphi_{X_{L}}\left(t\right) =\ex{e^{it X_{L}}}=\prod_{m=0}^{\infty}\ex{e^{\left(it a_{m,L}\left(\Abs{\z_m}^2-1\right)\right)}}=\prod_{m=0}^{\infty}\frac{e^{-ita_{m,L}}}{1-ia_{m,L}t}.
\end{equation*}
If $L=0$ then we recognise Weierstrass's definition of the Gamma function and get
\begin{equation*}
    \varphi_{X_0}\left(t\right)=e^{-i\g_e t}\G(1-it)
\end{equation*}
which yields Theorem~\ref{thm: tail}~\eqref{thm: tail Gumbel part}. For the rest of this section we assume that $0<L<\frac12$ and write $\varphi_{X_{L}}\left(t\right) =1/\Psi_L(-it)$
where $\Psi_{L}$ is the entire function
\begin{equation*}
    \Psi_{L}(z)= \prod_{m=0}^{\infty} \left(1+\frac{z}{b_{m,L}}\right) e^{-z/b_{m,L}};
\end{equation*}
here
\[
b_{m,L} = \frac1{a_{m,L}}=\frac{\Gamma(L)m!}{\Gamma(L+m)}
\]
which satisfies $b_{m,L}\sim \Gamma(L) m^{1-L}$ as $m\to\infty$. This allows us to explicitly describe the tail behaviour of $X_L$.

Denote by $\mathfrak{n}_L(R)$ the zero counting function of $\Psi_L$ so that
\[
\mathfrak{n}_L(R) = \# \{ z \in \C \colon \Psi_L(z) = 0 \mbox{ and } |z| \le R \} = \# \{ m \in \N \colon b_{m,L} \le R \} \sim \left(\frac{R}{\Gamma(L)}\right)^{1/(1-L)}
\]
for fixed $L$. We put $\rho = \frac1{1-L} \in (1,2)$ and $d_\rho = \Gamma(1 - 1/\rho)^{-\rho} = \Gamma(L)^{-1/(1-L)}$, and notice that $\Psi_L$ is an entire function of order $\rho$ and genus $1$, which corresponds to the fact that
\[
\sum_{m = 1}^\infty \frac{1}{b_{m,L}^2} < \infty.
\]

\subsection{Preliminaries: growth bounds for  \texorpdfstring{$\Psi_L$}{GL}}
We will deduce Theorem~\eqref{thm: tail}~\eqref{thm: tail L>0 part} by standard techniques which require growth estimates for $\Psi_L$. The estimates that we will use are given in the following proposition.

\begin{prop}\label{prop: growth psi}\phantom{w}
\begin{rlist}
    \item \label{Psi_L_pos_axis_asymp} As $R\to\infty$
    \[
    \log \Psi_L(R) = \frac{\pi d_\rho }{\sin(\pi \rho)} R^\rho (1+o(1)).
    \]
    
    \item \label{lower_bnd_Psi_L} Fix $x\le0$. If $|y| \ge 2|x|$ is sufficiently large then
    \[
    \log|\Psi_L(x+iy)| \ge c |y|^\rho,
    \]
    where $c > 0$ is a constant that depends only on $\rho$.
\end{rlist}
\end{prop}
\begin{rmk}
The estimate given in \eqref{lower_bnd_Psi_L} is rather crude, but suffices for our purposes.
\end{rmk}

\begin{proof}
Part \eqref{Psi_L_pos_axis_asymp} is immediate from \cite{regVarBook}*{Equation (7.2.3)}. To see part \eqref{lower_bnd_Psi_L} we note that, with the definitions given above, we have
\[
\log|\Psi_L(z)| = - \integlim{0}{\infty}{\Re{ \frac{z^2}{t+z} } \frac{\mathfrak{n}_L(t)}{t^2}}{t}
\]
for $\arg z \ne \pi$, by \cite{regVarBook}*{Theorem 7.2.1}. Now fix $x\le0$ and note that if $|y| \ge |x|$ then
\[
-\Re{ \frac{z^2}{t+z} } = \frac{y^2(t-x)-x^2 (t+x)}{(t+x)^2+y^2}\ge 0.
\]
Further if $t \le |y| + |x|$ and $|y| \ge 2 |x|$ then we have the lower bound
\[
\frac{y^2(t-x)-x^2 (t+x)}{(t+x)^2+y^2} \ge \frac{t(y^2-x^2)}{2y^2} \ge \frac t4.
\]
We conclude that if $|y|$ is sufficiently large then there is a constant $c = c_\rho > 0$ such that
\[
\log|\Psi_L(z)| \ge c \integlim{1}{|y|+|x|}{t^{\rho - 1}}{t} \ge c^\prime |y|^\rho.\qedhere
\]
\end{proof}

\subsection{The left tail}
We will deduce the asymptotics for the left tail from a Tauberian theorem of Kasahara.

\begin{prop}[\cite{HM}*{Lemma 3}]\label{Kasahara_left_tail}
Let $\rho\in(1,2)$ and let $Z$ be a random variable satisfying $\pr{Z<a}>0$ for any $a\in\R$, then
\[
\lim_{\lambda\to\infty} \frac{1}{\lambda^{\rho}} \log\ex{e^{-\lambda Z}} = A>0
\]
if and only if
\[
\lim_{y\to\infty} \frac{1}{y^{\rho/\left(\rho-1\right)}} \log\pr{Z<-y} = -\left(1-\frac{1}{\rho}\right) \left(\frac{1}{\rho A}\right)^{\frac{1}{\rho-1}}.
\]
\end{prop}

\begin{rmk}
Kasahara originally stated his result for the right tail of a random variable $X$ that satisfies $\pr{X>a}>0$ for every  $a>0$ but, as noted in the remark after \cite{HM}*{Lemma 3}, this assumption is easy to remove.
\end{rmk}

By Proposition~\ref{prop: growth psi}~\eqref{Psi_L_pos_axis_asymp}
\[
\lim_{\lambda\to\infty} \frac{1}{\lambda^\rho} \log \ex{e^{- \lambda X_{L}}} = -\lim_{\lambda\to\infty} \frac{1}{\lambda^\rho} \log \Psi_L(\l)= - \frac{\pi d_\rho}{\sin(\pi \rho)}
\]
and so, by Proposition~\ref{Kasahara_left_tail} we find that

\begin{align*}
\log\pr{X_L<-y} & \sim -\left(1-\frac{1}{\rho}\right)\left(\frac{-\sin(\pi \rho)}{\pi \rho d_\rho}\right)^{\frac{1}{\rho-1}} y^{\rho/(\rho - 1)}\\
& = - L \, \Gamma(L)^{1/L} \left(-\sinc\left(\frac{\pi}{1-L}\right)\right)^{\frac{1}{L}-1} y^{1/L} ,
\end{align*}
as $y \to \infty$, where $\sinc x = \sin x / x$, which is one half of Theorem~\ref{thm: tail}~\eqref{thm: tail L>0 part}.

\subsection{The right tail}
For the right tail we first use the inversion formula to compute the density $f_{X_{L}}$ of the random variable $X_L$. We have
\[
f_{X_{L}}(x)
= \frac{1}{2\pi} \integllim{\R}{\varphi_{X_{L}} (t) e^{-ixt}}{t} 
= \frac{1}{2\pi} \integllim{\R}{\frac{1}{\Psi_L(-it)}e^{-ixt}}{t},
\]
which holds by Proposition~\ref{prop: growth psi}~\eqref{lower_bnd_Psi_L}. By the same estimate, we can make a change of contour, and thus by the residue theorem,
\[
f_{X_{L}}(x)
= \frac{1}{2\pi} \integllim{{\R - i \b}}{\psi(w)}{w} - i \,\text{Res}(\psi,-i b_{0,L}),
\quad \mbox{ where }
\psi(w) = \frac{1}{\Psi_L(-i w)}e^{-ixw},
\]
and $\b \in (b_{0,L}, b_{1,L}) = (1, L^{-1})$ is some constant (the sign of the residue is negative since the contour runs clockwise). Again by Proposition~\ref{prop: growth psi}~\eqref{lower_bnd_Psi_L},
\[
\left| \frac{1}{2\pi} \integllim{{\R - i \b}}{\psi(w)}{w} \right| \le C e^{-\b x} \integllim{\R}{e^{-c |t|^\rho}}{t} \le C e^{-\b x}.
\]
It remains to evaluate the residue. We have $\text{Res}(\psi,-i)= e^{-x} \text{Res}(\frac1{\Psi_L(-i\,\cdot\,)},-i)$ and we write
\[
\frac1{\Psi_L(-iw)} = \frac{e^{-i w}}{1-iw} \prod_{m=1}^{\infty} \frac{e^{-i w/b_{m,L}}}{1-\frac{iw}{b_{m,L}}}
\]
so that
\begin{align*}
-i \,\text{Res}(\psi,-i b_{0,L}) = -i \,\text{Res}(\psi,-i) & = -ie^{-x} \prod_{m=1}^{\infty} \frac{e^{-1/b_{m,L}}}{1-\frac{1}{b_{m,L}}} \lim_{w \to -i} e^{-i w} \frac{w+i}{1-iw} \\
& = \frac1e \prod_{m=1}^{\infty} \frac{e^{-a_{m,L}}} {1-a_{m,L}} e^{- x} = \kappa_L e^{-x}.
\end{align*}
We conclude that, as $x\to\infty$,
\[
f_{X_{L}}(x) \sim \kappa_L e^{-x}
\]
(the error term is actually exponentially small) and in particular, as $y\to\infty$,
\[
\pr{X_L > y} \sim \kappa_L e^{-y}.
\]
This completes the proof of Theorem~\ref{thm: tail}~\eqref{thm: tail L>0 part}.
\appendix
\section{Wiener Chaos expansion for number of zeroes}\label{app: chaos}
In this appendix we prove versions of Proposition~\ref{prop: Wiener expan} and Lemma~\ref{lem: 2nd/4th moment chaos}.
We return to the notation introduced in Section~\ref{sec: var comput}; we define a GAF of the form
\[
f(z) = \sum_{m\ge 0} b_m \zeta_m z^m
\]
with $b_m \ge 0$ and covariance function
\[
G(z) = \sum_{m \ge 0} b_m^2 z^m
\]
and write $\widehat{f}\left(z\right)=\frac{f(z)}{G(|z|^2)^{1/2}}$. We also write $n(r)$ for the number of zeroes of $f$ in the disc $D(0,r)$ for any $r<R_0$, where $R_0$ denotes the radius of convergence of $G$ (which is a.s. the radius of convergence of $f$, and we allow $R_0=\infty$).

We next recall the notion of the Wiener chaos. We define the $q$-th component of the Wiener chaos to be
\begin{equation*}
    \mathcal{W}^{\wik{q}} = L^2-\spn \{ \wikkk{\z_{j_1}}{\a_1 } {\z_{j_1}} {\b_1} \cdots \wikkk{\z_{j_k}}{\a_k}{\z_{j_k}}{\b_k} \mid \a_1+\b_1 + \cdots + \a_k+\b_k=q\}.
\end{equation*}
It follows from \cite{Jan}*{Theorem 3.12} that the $\mathcal{W}^{\wik{q}}$ define orthogonal subspaces of $L^2$, and from \cite{Jan}*{Theorem 2.6} we get that $L^2=\bigoplus_{q=0}^\infty \mathcal{W}^{\wik{q}}$. Given any random variable with finite second moment, we may therefore expand it in terms of its projection to each $\mathcal{W}^{\wik{q}}$, and this is known as the Wiener chaos expansion. We now state this expansion for $n(r)$.

\begin{prop}\label{prop: gen chaos expan}
Define
\begin{equation}\label{eq: expansion for n a}
    n(r;\a)=\frac{\left(-1\right)^{\a+1}}{\a\left(\a!\right)}\frac{1}{2\pi i}\integlim{\partial D\left(0,r\right)}{}{\frac{\partial}{\partial z}\wikk{\widehat{f}\left(z\right)}{\a}}{z}.
\end{equation}
Then $n(r;\a)$ belongs to the $2\a$-th component of the Wiener chaos corresponding to $f$ and 
\begin{equation*}
    n(r)-\ex{n(r)}=\sum_{\a=1}^\infty n(r;\a)
\end{equation*}
where the sum converges in $L^2$.
\end{prop}
We will also need the following lemma, which generalises Lemma~\ref{lem: 2nd/4th moment chaos} above.
\begin{lem}\label{lem: gen 2nd/4th mom}
  If $\a\ge1$ then
  \begin{equation*}
      \ex{\big(n(r;\a)\big)^{2}} = \left(\frac{1}{2\pi i}\right)^{2} \frac{1}{\a^{2}} \iint\limits_{\partial D\left(0,r\right)^{2}} \frac{\partial^{2}} {\partial z\partial w} \frac{\left|G\left(z\overline{w}\right)\right|^{2\alpha}}{G(|z|^2)^\a G(|w|^2)^\a} \mathrm{d}z \mathrm{d}w
  \end{equation*}
  and
  \[
  \ex{\big(n_{L}(r;\alpha)\big)^{4}} = \left(\frac{1}{2\pi i}\right)^{4} \frac{1}{\alpha^{4}\left(\alpha!\right)^{4}} \int_{\partial D\left(0,r\right)^{4}} \frac{\partial^{4}} {\partial z_{1}\dots\partial z_{4}} \ex{\prod_{j=1}^{4}\wikk{\widehat{f}\left(z_{j}\right)} {\alpha} }\prod_{j=1}^{4}\mathrm{d}z_{j}.
  \]
\end{lem}

To prove the above results we require the following lemmas, which allow us to justify interchanging the order of some operations.
\begin{lem}[cf \cite{BS}*{Lemma 7}]\label{lem: pth mom grad}Given a polynomial $P$, $R<R_0$ and $1\le p<+\infty$
we have
\[
\E\left[\left|\nabla P\left(\left|\widehat{f}\left(z\right)\right|^{2}\right)\right|^{p}\right]\le C\left(p,R,P,G\right)
\]
for $z\in D\left(0,R\right)$.
\end{lem}
\begin{lem}[cf \cite{BS}*{Lemma 8}]\label{lem: exchange}
  Let $\psi_j:\R_+\to\R$ be differentiable functions for $1\leq j\leq N$ and let $\Psi_j=\psi_j\circ \Abs{\widehat{f}}^2$. Suppose that
  \begin{equation}\label{eq: pth mom grad Lj}
    \int_{\partial D\left(0,r\right)^N} \E\Big[ \Big| \prod_{j=1}^N \nabla\Psi_j(z_j) \Big| \Big] \prod_{j=1}^N |\mathrm{d}z_j| <+\infty
  \end{equation}
  and that, for almost every tuple $(z_1,\dots,z_N)$ with respect to the measure $\prod_{j=1}^N |\mathrm{d}z_j|$, there exists $\e_0>0$ and $1<p<2$ such that
  \begin{equation}\label{eq: bdd grad p mom}
    \sup_{\forall j:\,w_j\in D(z_j,\e_0)}\E\Big[ \Big| \prod_{j=1}^N \nabla\Psi_j(w_j) \Big|^p \Big]<+\infty.
  \end{equation}
  Then
  \begin{equation*}
    \E\bigg[ \int_{\partial D\left(0,r\right)^N} \frac{\partial^N}{\partial z_1\cdots\partial z_N} \prod_{j=1}^N  \Psi_j(z_j) \prod_{j=1}^N |\mathrm{d}z_j| \bigg]
    =\int_{\partial D\left(0,r\right)^N} \frac{\partial^N}{\partial z_1\cdots\partial z_N} \ex{\prod_{j=1}^N \Psi_j(z_j) }\prod_{j=1}^N |\mathrm{d}z_j|.
  \end{equation*}
\end{lem}
\begin{lem}[cf \cite{BS}*{Lemmas 9, 10 and 11}]\label{lem: hypoth sat polys}
  \phantom{a}
  \begin{alist}
  \item Suppose that $\psi_j$ are polynomials for $1\leq j\leq N$. Then \eqref{eq: pth mom grad Lj} and \eqref{eq: bdd grad p mom} hold.
  \item Suppose that $N=2$, $\psi_1=\log$ and $\psi_2$ is a polynomial. Then \eqref{eq: pth mom grad Lj} and \eqref{eq: bdd grad p mom} hold.
  \item Suppose that $N=2$ and $\psi_1=\psi_2=\log$. Then \eqref{eq: pth mom grad Lj} holds and for every pair $(z_1,z_2)$ with $z_1\neq z_2$, \eqref{eq: bdd grad p mom} holds.
  \end{alist}
\end{lem}

The proof of Lemma~\ref{lem: pth mom grad} is postponed until later. The proofs of Lemmas~\ref{lem: exchange} and~\ref{lem: hypoth sat polys} are essentially identical to the proofs given in \cite{BS}, and are accordingly omitted. Combining Lemmas~\ref{lem: exchange} and~\ref{lem: hypoth sat polys} with \cite{Jan}*{Theorem 3.9} immediately yields Lemma~\ref{lem: gen 2nd/4th mom}. We now proceed to prove Proposition~\ref{prop: gen chaos expan}

\begin{proof}[Proof of Proposition~\ref{prop: gen chaos expan}]
We first show that the random variable $n\left(r;\a\right)$ defined in \eqref{eq: expansion for n a} belongs to $\mathcal{W}^{\wik{2\a}}$. Notice that Lemma~\ref{lem: pth mom grad} implies that
\begin{equation}
\E\left[\left|\nabla\wikk{\widehat{f}\left(z\right)}{\alpha}\right|^{2}\right]\le C\left(R,\alpha,G\right)\label{eq: grad est}
\end{equation}
for $z\in D\left(0,R\right)$. Standard arguments show that $\wikk{\widehat{f}\left(z\right)}{\alpha}$
is in $\mathcal{W}^{\wik{2\a}}$, and therefore so is $\frac{\wikk{\widehat{f}\left(z+h\right)}{\alpha}-\wikk{\widehat{f}\left(z\right)}{\alpha}}{h}$
for any $h\in\C$ such that $|z+h|<R_0$. Taking real $h\to0$ and applying the mean value
theorem and \eqref{eq: grad est}, we use dominated convergence to
see that $\frac{\partial}{\partial x}\wikk{\widehat{f}\left(z\right)}{\alpha}$
is in $\mathcal{W}^{\wik{2\a}}$ for any fixed $z\in D(0,R_0)$. Arguing similarly for
imaginary $h$ we see that $\frac{\partial}{\partial z} \wikk{\widehat{f}\left(z\right)}{\alpha}$
is in $\mathcal{W}^{\wik{2\a}}$. Now write $g_{\alpha}\left(z\right)= \frac{\partial}{\partial z} \wikk{\widehat{f}\left(z\right)}{\alpha}$
and consider the Riemann sum $\frac{1}{N}\sum_{j=1}^{N}g\left(re^{2\pi i\frac{j}{N}}\right)$
which is in $\mathcal{W}^{\wik{2\a}}$. Then
\[
\ex{\bigg| \frac{1}{N} \sum_{j=1}^{N}g\left(re^{2\pi i\frac{j}{N}}\right) \bigg|^{2}} \le \frac{1}{N} \sum_{j=1}^{N} \ex{\left|g\left(re^{2\pi i\frac{j}{N}}\right) \right|^{2}} \le C\left(R,\alpha,G\right)
\]
and once more applying dominated convergence we see that the Riemann
sums converge to the integral
\[
\integlim{0}{\pi}{g\left(re^{i\theta}\right)}{\theta}=\frac{1}{ir}\integlim{\partial D\left(0,r\right)}{}{\frac{\partial}{\partial z}\wikk{\widehat{f}\left(z\right)}{\alpha}}{z}
\]
in $L^2$ and so $n(r;\a)$ belongs to $\mathcal{W}^{\wik{2\a}}$ as claimed.

It remains to prove that $\sum_{\a=1}^{M}n\left(r;\a\right)\to n\left(r\right)-\ex{n\left(r\right)}$
in $L^{2}$ as $M\to\infty$. The basic strategy is to implement the scheme outlined in Section~\ref{sec: overview}, and to use the lemmas above to justify the steps. We will briefly outline the argument, which closely follows  \cite{BS}*{Section 3.2}.

We define
\begin{equation*}
    \log_M\left|\z\right|^{2}=-\g_e+\sum_{\a=1}^{M}\frac{(-1)^{\a+1}}{\a(\a!)}\wikk{\z}{\a}
\end{equation*}
to be the truncation of the series in \eqref{eq: log expan}. Notice that
\begin{equation*}
    n\left(r\right)-\ex{n\left(r\right)}-\sum_{\a=1}^{M}n\left(r;\a\right)=\frac{1}{2\pi i}\integlim{\partial D\left(0,r\right)}{}{\frac{\partial}{\partial z}\left(\log\Abs{\widehat{f}_L(z)}^{2}-\log_M \Abs{\widehat{f}_L(z)}^{2}\right)}{z}.
\end{equation*}
We square this expression and take its expectation. Appealing to Lemmas~\ref{lem: exchange} and~\ref{lem: hypoth sat polys} we can exchange the order of operations to get
\begin{align*}
\mathcal{E}_M&\coloneqq\E\left[\left(n\left(r\right)- \ex{n\left(r\right)} -\sum_{\alpha=1}^{M} n\left(r;2\alpha\right) \right)^{2}\right]\\
&=-\frac{1}{4\pi^{2}}\iint\limits_{\partial D\left(0,r\right)^{2}} \frac{\partial^{2}} {\partial z\partial w}\ex{\left(\log\Abs{\widehat{f}_L(z)}^{2}-\log_M \Abs{\widehat{f}_L(z)}^{2}\right)\left(\log\Abs{\widehat{f}_L(w)}^{2}-\log_M \Abs{\widehat{f}_L(w)}^{2}\right)} \mathrm{d}z\,\mathrm{d}w\\
&=-\frac{1}{4\pi^{2}}\iint\limits_{\partial D\left(0,r\right)^{2}} \frac{\partial^{2}} {\partial z\partial w}\sum_{\a,\b>M}\frac{(-1)^{\a+\b}}{\a\b(\a!\b!)}\ex{\wikk{\widehat{f}_L(z)}{\a}\,\wikk{\widehat{f}_L(w)}{\b}} \mathrm{d}z\,\mathrm{d}w\\
&=-\frac{1}{4\pi^{2}} \sum_{\alpha>M} \frac{1}{\alpha^{2}}\iint\limits_{\partial D\left(0,r\right)^{2}} \frac{\partial^{2}} {\partial z\partial w} \left|\frac{G(z\overline{w})}{\sqrt{G(|z|^2)G(|w|^2)}}\right|^{2\alpha}\mathrm{d}z\,\mathrm{d}w,
\end{align*}
where the final equality follows from \cite{Jan}*{Theorem 3.9}. Simplifying the integrand à la \cite{KN}*{Claim A.2} we get
\begin{equation}\label{eq: tail L^2}
\mathcal{E}_M=\mathcal{E}_M(r)=\frac{1}{2\pi} \sum_{\alpha>M} \integlim{-\pi}{\pi}{\Abs{\frac{G(r^2e^{i\t})} {G(r^2)}}^{2\a} \Abs{ \frac{G'(r^2e^{i\t}) r^2 e^{i\t}} {G(r^2e^{i\t})} - \frac{G'(r^2)r^2} {G(r^2)} }^2 } {\t}
\end{equation}
which is the tail of a convergent sum.
\end{proof}
\begin{rmk}
Putting $M=0$ in \eqref{eq: tail L^2} is a recasting of \eqref{eq: varformula}.
\end{rmk}

\begin{proof}[Proof of Lemma \ref{lem: pth mom grad}]
We mimic the proof of \cite{BS}*{Lemma 7}. It suffices to show that
\[
\ex{\left|\nabla\left(\left|\widehat{f}\left(z\right)\right|^{2}\right)\right|^{p}} \le C\left(p,R,G\right)
\]
for $z\in D\left(0,R\right)$. Now
\[
\left|\nabla\left|\widehat{f}\left(z\right)\right|^{2}\right|\lesssim\frac{|f'(z)f(z)|}{G\left(\left|z\right|^{2}\right)}+\frac{|z|G'\left(\left|z\right|^{2}\right)}{G\left(\left|z\right|^{2}\right)}|\widehat{f}(z)|^{2}
\]
and since $\ex{|\widehat{f}(z)|^{2p}}$ is independent of
$z$ and $G$ is analytic and zero-free on $[0,1)$ we see that
\[
\ex{\left|\frac{|z|G'\left(\left|z\right|^{2}\right)}{G\left(\left|z\right|^{2}\right)}|\widehat{f}(z)|^{2}\right|^{p}}\le C\left(R,p,G\right).
\]
Similarly
\[
\ex{\left|\frac{|f'(z)f(z)|}{G\left(\left|z\right|^{2}\right)}\right|^{p}} \le \ex{\frac{|f'(z)|^{2p}} {G\left(\left|z\right|^{2}\right)^{p}}} ^{\frac{1}{2}} \ex{|\widehat{f}(z)|^{2p}}^{\frac{1}{2}} \le C\left(R,p,G\right) \ex{\frac{|f'(z)|^{2p}} {G\left(\left|z\right|^{2}\right)^{p}}}^{\frac{1}{2}}.
\]
Now $f'\left(z\right)$ is a Gaussian random variable with variance
$G'\left(\left|z\right|^{2}\right)+\left|z\right|^{2}G''\left(\left|z\right|^{2}\right)$
and so
\[
\ex{\frac{|f'(z)|^{2p}} {G\left(\left|z\right|^{2}\right)^{p}}} = \Gamma\left(1+\frac{p}{2}\right) \left(\frac{G'\left(\left|z\right|^{2}\right) +\left|z\right|^{2}G''\left(\left|z\right|^{2}\right)} {G\left(\left|z\right|^{2}\right)}\right)^{p}
\]
which is bounded as before.
\end{proof}

\section{Standard lemma to deduce CLT}\label{app: CLT}
In this appendix we prove that the scheme outlined in Section~\ref{sec: L>1/2} indeed implies a CLT.
\begin{lem}
  Let $X_n$ and $X_n(M)$, for $n>0$ and $M>0$, be real-valued random variables with mean $0$ and variance $1$. Suppose that the following holds:
\begin{itemize}
  \item For each fixed $M$,
  \[
  X_n(M) \overset{d}{\longrightarrow} \cN_\R(0,1) \quad \text{as} \quad n\to\infty.
  \]
  \item We have
  \[
  \lim_{M\to\infty}\ex{(X_n-X_n(M))^2} =0,
  \]
  uniformly in $n$.
\end{itemize}
Then
\[
 X_n \overset{d}{\longrightarrow} \cN_\R(0,1) \quad \text{as} \quad n\to\infty.
  \]
\end{lem}
We will give a short proof using characteristic functions, though one can also give an elementary direct proof.
\begin{proof}
We write $\varphi(t)=e^{-t^2/2}$,
\begin{equation*}
    \phi_n(t)=\ex{e^{itX_n}}\quad\text{and}\quad \phi_{n,M}(t)=\ex{e^{itX_n(M)}}.
\end{equation*}
We estimate, for fixed $t$,
\begin{align*}
    \abs{\phi_n(t)-\varphi(t)}&\le \E\big|e^{it(X_n-X_n(M))}-1\big|+\abs{\phi_{n,M}(t)-\varphi(t)}\\
    &\le |t| \E\big|X_n-X_n(M)\big|+\abs{\phi_{n,M}(t)-\varphi(t)}.
\end{align*}
Applying the hypothesis we see that $\phi_n$ converges pointwise to $\varphi$, which implies the result.
\end{proof}

\end{document}